\documentclass[11pt,a4paper,reqno]{article}
\usepackage[margin=1in]{geometry}
\usepackage[hidelinks]{hyperref}
\usepackage{amsmath,amssymb,amsfonts}
\usepackage{amsthm}
\usepackage{mathtools}
\usepackage{bm}
\usepackage{bbm}
\usepackage{graphicx}
\usepackage{enumitem}
\usepackage{hyperref}
\usepackage{mathrsfs}
\usepackage{color}
\usepackage{algorithm}
\usepackage{algorithmic}

\newtheorem{theorem}{Theorem}[section]
\newtheorem{lemma}[theorem]{Lemma}
\newtheorem{proposition}[theorem]{Proposition}
\newtheorem{corollary}[theorem]{Corollary}
\newtheorem{definition}[theorem]{Definition}

\newtheorem{assumption}[theorem]{Assumption}

\newcommand{\R}{\mathbb{R}}

\newcommand{\E}{\mathbb{E}}

\newcommand{\Cov}{\operatorname{Cov}}
\newcommand{\KL}{\mathrm{D}_{\mathrm{KL}}}

\newcommand{\tr}{\operatorname{tr}}

\title{A Kernel Formula for  Kinetic Fokker-Planck Equations}

\author{
  Fuqun Han\thanks{Department of Mathematics, City University of Hong Kong, Hong Kong. \texttt{fuqun.han@cityu.edu.hk}.}
  \and
  Wuchen Li\thanks{Department of Mathematics, University of South Carolina, Columbia, SC, USA. \texttt{wuchen@mailbox.sc.edu}.  W. Li’s work is supported by the AFOSR YIP award No. FA9550-23-1-0087, NSF RTG: 2038080, NSF FRG DMS-2245097, NSF DMS-2601952 and the McCausland Faculty Fellow in University of South Carolina.}
}
\date{}
\begin{document}
\maketitle
\begingroup
\renewcommand{\thefootnote}{}
\footnotetext{The authors thank Professor Stanley Osher for many helpful discussions during the development of this work.}
\endgroup
\begin{abstract}
We derive an explicit one-step kernel formula for the kinetic Fokker-Planck equation. The construction begins with a reference dynamics admitting an explicit transition kernel and uses a carefully chosen auxiliary weight to encode the drift perturbation. The resulting kernel formula provides an explicit approximation of the density evolution. In a weighted phase-space setting, we establish local weak consistency and first-order finite-time weak convergence under suitable regularity and moment assumptions. We also give a formal variational interpretation of the formula in terms of a regularized Wasserstein-type proximal operator and present analogous constructions for a broader class of parabolic equations. Numerical experiments illustrate the accuracy of the kernel approximation and its application to deterministic particle schemes for kinetic sampling.
\end{abstract}

\section{Introduction}

Fokker-Planck equations describe the evolution of probability distributions associated with stochastic processes and interacting particle systems \cite{risken1996fokker}. They arise naturally in settings where the distribution itself is the main object. In machine learning, they appear in score-based generative modeling and diffusion models, where they describe the evolution of data distributions \cite{song2020score,ho2020denoising}. They also arise as forward equations governing agents' state distributions in mean field games and mean field control problems \cite{ruthotto2020machine,lasry2007mean}. Moreover, Fokker-Planck equations arise as mean field descriptions of interacting-particle methods, including consensus-based global optimization and ensemble methods for inverse problems \cite{pinnau2017consensus,calvello2025ensemble}.

A fundamental structural viewpoint comes from optimal transport. In the classical reversible setting, the JKO scheme identifies the Fokker-Planck equation as the Wasserstein gradient flow of the free energy \cite{jordan1998variational}. Closely related dynamic, control-theoretic, and entropic formulations are provided by the Benamou-Brenier formulation and Schr\"odinger bridge problems \cite{benamou2000computational,leonard2014survey,ambrosio2005gradient}.  These formulations interpret the dynamics as transport at the level of densities with entropy dissipation, and they motivate numerical schemes that preserve this structure. However, this structural formulation does not by itself resolve the computational challenges of solving Fokker-Planck equations in high dimensions. The direct implementation of variational schemes such as the JKO scheme typically requires solving a high-dimensional optimal transport problem at each time step. Grid-based solvers suffer from the curse of dimensionality. Particle methods are often more scalable, but they may introduce sampling noise, lose resolution in low-density regions, and provide only indirect access to the evolving density and its derivatives. 

These limitations have motivated a range of density level and mesh free approaches. Examples include machine learning methods for high-dimensional mean field games and control \cite{ruthotto2020machine}, inexact or learned approximations of JKO-type schemes \cite{di2025inexact,lee2024deepjko,xu2023normalizing}, and continuous transport methods such as probability flow formulations, flow matching, and stochastic interpolants \cite{boffi2022probability,lipman2022flow,albergo2025stochastic}.
Beyond the classical Wasserstein gradient flow setting, non-gradient, anisotropic, and degenerate Fokker-Planck equations have also been studied through formulations that separate reversible transport from irreversible dissipation, including the general equation for nonequilibrium reversible and irreversible coupling (GENERIC) framework \cite{ottinger2005beyond,duong2013generic}.

The kinetic Fokker-Planck equation is a particularly important example beyond the Wasserstein gradient flow setting. It is the forward equation of underdamped Langevin dynamics, which is widely used in molecular dynamics, free energy sampling, and momentum based Monte Carlo methods \cite{lelievre2016partial,betancourt2017conceptual,cheng2018underdamped}. Its structure is more delicate than that of the overdamped Langevin equation: transport acts in the position variable, whereas diffusion acts only in the velocity variable. This hypoelliptic coupling creates both analytical and numerical challenges \cite{villani2009hypocoercivity}.  A natural question arises:  Is there an explicit numerical update that is computationally tractable while still reflecting the reversible-irreversible structure of the underlying kinetic Fokker-Planck equation?

This paper constructs an explicit kernel operator for approximating the kinetic Fokker-Planck equation. We start from the kernel formula introduced in \cite{li2023kernel} for the regularized Wasserstein proximal operator. As in \cite{han2024convergence}, this kernel formula yields an approximation to the evolution of the reversible Fokker-Planck equation. We now introduce an explicit kernel formula for the kinetic Fokker-Planck equation
\[
\partial_t\rho +v\cdot\nabla_x\rho - \nabla_v\cdot (\nabla_x V(x)\rho)
= \nabla_v\cdot\!\Bigl(\gamma v\rho+\gamma\beta^{-1}\nabla_v\rho\Bigr),
\]
where $x\in\R^d$ and $v\in\R^d$ denote the position and velocity variables, respectively, $V:\R^d\to\R$ is the potential energy, and $\gamma>0$ and $\beta>0$ are the friction coefficient and inverse temperature. 
In our setting, the construction leads to the explicit one-step operator
\begin{equation}
\label{intro:kernel}
(\mathcal K_h\rho_0)(x,v) := \psi(x,v) \int_{\R^{2d}} \frac{G_h(x,v\mid y,w)\,\rho_0(y,w)} {(G_h\psi)(y,w)}\,dy\,dw,
\end{equation}
where
\[
G_h(x,v\mid y,w)=\frac{3^{d/2}(\beta/\gamma)^d}{(2\pi h^2)^d}\exp\!\left(-\frac{3\beta}{\gamma h^3}\left\|x-y-\frac{h}{2}(v+w)\right\|^2-\frac{\beta}{4\gamma h}\|v-w\|^2\right).
\]
Here $\psi = \exp(-\phi/2)$ with $\phi =  \lambda V(x)+\frac{\beta}{\gamma}v\cdot\nabla_xV(x)+\frac{\beta}{2}|v|^2$ and $\lambda>0$. 
For sufficiently small $h$ in the regime where $0<(G_h\psi)(y,w)<\infty$, the above kernel formula provides a direct, explicit approximation of the density evolution. By differentiating the kernel representation, it also yields a practical estimator for the associated score function, which is the gradient of the log density function.

The kernel formula admits two complementary derivations. The first derivation separates the generator into a reference part with an explicit transition kernel and a drift perturbation. The drift perturbation is then incorporated through an auxiliary weight and a normalization of the reference semigroup. The second comes from the variational formulation. We formulate a constrained optimal control problem with Fisher information regularization and derive the same kernel formula \eqref{intro:kernel} from the associated forward-backward optimality system through a Hopf-Cole transformation. Formally, this variational formulation can be interpreted as a kinetic analog of a regularized Wasserstein proximal step.

Existing variational approaches have clarified important transport and dissipation structures of kinetic and non-gradient Fokker-Planck equations \cite{huang2000variational,duong2014conservative,duong2022entropic,carrillo2021variational, kinetic_OT}. Their numerical realizations often lead to implicit time discretizations or require auxiliary variational subproblems. By contrast, our construction provides an explicit one-step kernel operator for the kinetic Fokker-Planck equation. To the best of our knowledge, a closed form kernel formula of this type has not previously been proposed in the kinetic setting.

The paper is organized as follows. In Section~\ref {sec:kernel_kinetic}, we derive the kernel formula for the kinetic Fokker-Planck equation. We then prove the corresponding weak approximation results in weighted phase space and provide a variational interpretation. Section~\ref{sec:generalizations} records the analogous constructions in more general degenerate and nondegenerate parabolic settings. Finally, Section~\ref{sec:ne} presents numerical experiments using the kernel formula for density evolution and kinetic sampling problems.

\section{Kernel Formula and Variational Structure in the Kinetic Setting}
\label{sec:kernel_kinetic}

In this section, we derive the kernel formula for the kinetic Fokker-Planck equation, first from operator splitting through a reference semigroup and then from a quadratic optimal control problem via a Hopf-Cole transformation.

We consider the kinetic Fokker-Planck equation on phase space
$(x,v)\in\R^d\times\R^d$:
\begin{equation}\label{eq:kinetic-FP}
\partial_t\rho +v\cdot\nabla_x\rho - \nabla_v\cdot (\nabla_x V(x)\rho)= \nabla_v\cdot\!\Bigl(\gamma v\rho+\gamma\beta^{-1}\nabla_v\rho\Bigr),
\end{equation}
where $t>0$, $V:\R^d\to\R$ is the potential energy, $\gamma>0$ is the friction coefficient, and $\beta>0$ is the inverse temperature. In the following, we will write
\[
z=(x,v)\in\R^{2d},
\qquad
U(x,v):=\nabla_xV(x)+\gamma v.
\]
 
To allow unbounded potentials, we work in weighted spaces to handle the tail terms. Let $W:\R^{2d}\to[1,\infty)$ be a weight function to be defined later. Define the weighted space of measurable functions by
\[
L_W^\infty(\R^{2d}) := \left\{ g:\R^{2d}\to\R:\ \|g\|_{L_W^\infty}:= \sup_{(x,v)\in\R^{2d}}\frac{|g(x,v)|}{W(x,v)}<\infty \right\}.
\]
For \(k\in\mathbb N\), we define the weighted space of smooth functions by
\[
C_W^k(\R^{2d}):=\left\{f\in C^k(\R^{2d}) :\|f\|_{C_W^k}:=\sum_{|\alpha|\le k}\|\partial^\alpha f\|_{L_W^\infty}<\infty\right\}.
\]
Next, we introduce the following notation for the weighted moment of a density \(\rho\):
\[
m_W(\rho):=\int_{\R^{2d}}W(x,v)\rho(x,v)\,dx\,dv,
\]
and define the space of probability densities
\[
\mathcal P_W(\R^{2d}):= \left\{\rho\in L^1(\R^{2d}): \rho\ge0,\ \int_{\R^{2d}}\rho dx\,dv=1,\ m_W(\rho)<\infty\right\}.
\]
 
We now introduce the reference backward operator
\begin{equation}\label{eq:D0-kin}
D_0 f:=v\cdot\nabla_x f+\gamma\beta^{-1}\Delta_v f,\qquad f\in C_c^\infty(\R^{2d}),
\end{equation}
and its formal $L^2$-adjoint
\begin{equation}\label{eq:D0star-kin}
D_0^*\rho:=-v\cdot\nabla_x\rho+\gamma\beta^{-1}\Delta_v\rho.
\end{equation}

The operator $D_0$ is the backward Kolmogorov operator associated with the transport diffusion system
\begin{equation}\label{eq:base-SDE-kin}
dx_t=v_t\,dt,
\qquad
dv_t=\sqrt{2\gamma\beta^{-1}}\,dB_t,
\end{equation}
where $B_t$ is a standard Brownian motion. Hence $D_0$ generates an explicit Gaussian semigroup $(G_h)_{h\ge0}$. Solving \eqref{eq:base-SDE-kin} with initial condition $(x_0,v_0)=(y,w)$ shows that $(x_h,v_h)$ is Gaussian, with mean
\[
\E[x_h\mid x_0=y,v_0=w]=y+hw,\qquad \E[v_h\mid x_0=y,v_0=w]=w,
\]
and covariance
\[
\Cov\!\begin{pmatrix}x_h\\ v_h\end{pmatrix}=\gamma\beta^{-1}
\begin{pmatrix}
\frac{2}{3}h^3 I_d & h^2 I_d\\
h^2 I_d & 2h\,I_d
\end{pmatrix}.
\]
Therefore, its transition kernel is the Gaussian density
\begin{equation}\label{eq:G0-kin}
G_h(x,v\mid y,w)=\frac{3^{d/2}(\beta/\gamma)^d}{(2\pi h^2)^d}\exp\!\left(-\frac{3\beta}{\gamma h^3}\left\|x-y-\frac{h}{2}(v+w)\right\|^2-\frac{\beta}{4\gamma h}\|v-w\|^2
\right).
\end{equation}
In the following, we write $(G_h\psi)(y,w) = \int_{\R^{2d}}G_h(x,v|y,w)\psi(x,v)dxdv$.
 
To encode the drift term $U=\nabla_xV+\gamma v$ through an explicit kernel formula, we choose $\phi=\phi_{\mathrm{ki}}^\lambda$ that satisfies $\gamma\beta^{-1}\nabla_v\phi = U$. A convenient choice is
\begin{equation}\label{eq:phi-kin}
\phi_{\mathrm{ki}}^\lambda(x,v) := \lambda V(x)+\frac{\beta}{\gamma}v\cdot\nabla_xV(x)+\frac{\beta}{2}|v|^2,
\qquad \lambda>0.
\end{equation}
By construction, we have
\[
\gamma\beta^{-1}\nabla_v\phi_{\mathrm{ki}}^\lambda(x,v) = \gamma\beta^{-1}\left(\frac{\beta}{\gamma}\nabla_xV(x)+\beta v\right)=\nabla_xV(x)+\gamma v=U(x,v).
\]

Next, we introduce the perturbed backward operator
\begin{equation}\label{eq:L-kin}
\mathcal L f:=D_0f-U\cdot\nabla_v f,
\end{equation}
whose adjoint is
\begin{equation}\label{eq:Lstar-kin}
\mathcal L^*\rho=
D_0^*\rho+\nabla_v\cdot(U\rho)=-v\cdot\nabla_x\rho+\gamma\beta^{-1}\Delta_v\rho+\nabla_v\cdot\!\bigl((\nabla_xV+\gamma v)\rho\bigr).
\end{equation}
Hence \eqref{eq:kinetic-FP} can be rewritten compactly as $\partial_t\rho=\mathcal L^*\rho$.

We now state our assumption on $V$ and the phase space weight.
\begin{assumption}
\label{ass:kinetic-structure}
Assume that \(V\in C^3(\R^d)\) is bounded from below and satisfies
\begin{equation}\label{eq:V-growth-assumption}
\|D^2V\|_{L^\infty(\R^d)}<\infty,
\qquad
|D^\mu V(x)|\le C_\mu(1+|x|)^{m_\mu}, \quad  |\mu|=3,
\end{equation}
for some constants $C_\mu,m_\mu\ge0$. Fix $\lambda>0$ and define $\phi_{\mathrm{ki}}^\lambda$ by \eqref{eq:phi-kin}.
For $0<\theta<\bar\theta$, define the base and stronger phase space weights by
\begin{equation}\label{eq:W-kin}
W(x,v):=\exp\!\bigl(\theta(1+|x|^2+|v|^2)\bigr), \qquad W_+(x,v):=\exp\!\bigl(\bar\theta(1+|x|^2+|v|^2)\bigr).
\end{equation}
\end{assumption}

We remark that the assumptions on \(V\) are mainly used to obtain a simple weight \(W\) with quadratic exponent, which is convenient for the analysis. For more general potentials, similar arguments may apply with a different choice of \(W\). The essential requirements are the estimates in Lemmas~\ref{lem:weighted-ratio-bound-kin} and~\ref{lem:weighted-Taylor-psif-kin}. The numerical experiments below also suggest that Gaussian-mollified variants of the kernel framework can remain effective beyond this specific assumption class.

For convenience, we set
\begin{equation}\label{eq:psi-kin} \psi(x,v):=e^{-\frac{\phi_{\mathrm{ki}}^\lambda(x,v)}{2}}.
\end{equation}
The role of \(\psi\) is to produce the missing drift $-U\cdot \nabla_v f$ in the backward generator. This will be made precise in Lemma~\ref{lem:Doob-h-exact-kin}.

The construction of the kernel formula below relies on two short-time estimates for the reference semigroup $G_h$. The first lemma controls the ratio between $\psi$ and \(G_h\psi\). The second lemma gives a second-order expansion of \(G_h(\psi f)\) relative to \(\psi\). Their proofs are technical and use only the explicit Gaussian kernel together with the growth assumptions on \(V\), so we defer them to Appendix~\ref{app:proof}.

\begin{lemma}[Weighted ratio bound]
\label{lem:weighted-ratio-bound-kin}
Under Assumption~\ref{ass:kinetic-structure}, there exist $h_0>0$ and $C>0$ such that, for every $0<h\le h_0$, one has $0<(G_h\psi)(x,v)<\infty$ for all $(x,v)\in\R^{2d}$ and
\begin{equation}\label{eq:weighted-ratio-bound-kin}
\left\|\frac{G_h\psi}{\psi}\right\|_{L_W^\infty}+\left\|\frac{\psi}{G_h\psi}\right\|_{L_W^\infty}\le C.
\end{equation}
\end{lemma}
 
\begin{lemma}[Weighted Taylor expansion]\label{lem:weighted-Taylor-psif-kin}
For $m\ge 0$, set
\[
W_m(x,v):=(1+|x|+|v|)^m W(x,v), \qquad W_{m}^{+}(x,v):=(1+|x|+|v|)^m W_+(x,v).
\]
Under Assumption~\ref{ass:kinetic-structure}, there exist an integer $m\ge0$ and a constant $C>0$ such that, for every $0<h\le h_0$, we have
\begin{equation}\label{eq:weighted-Taylor-psif-kin}
\left\|\frac{G_h(\psi f)-\psi f-hD_0(\psi f)}{\psi}\right\|_{L_{W_{m}^{+}}^\infty}\le Ch^2\|f\|_{C_W^4}, \qquad f\in C_W^4(\R^{2d}).
\end{equation}
As a special case of \eqref{eq:weighted-Taylor-psif-kin} with $f= 1$, we obtain
\begin{equation}\label{eq:weighted-Taylor-psi-kin}
\left\|\frac{G_h\psi-\psi-hD_0\psi}{\psi}\right\|_{L_{W_{m}^{+}}^\infty}\le Ch^2.
\end{equation}
\end{lemma}
 
\subsection{Kernel formula from reference semigroup}
\label{subsec:kernel-splitting}

To derive the kernel formula, the key observation is that the effect of the perturbation term $U\cdot\nabla_v f$ in the backward generator in \eqref{eq:L-kin} can be captured by applying the reference semigroup to the weighted function \(\psi f\) and then normalizing by \(G_h\psi\). We emphasize that $\psi$ is generally different from the invariant state of the kinetic Fokker-Planck equation; here it serves as the weight that encodes the gradient perturbation.

We begin with the following algebraic identity and the associated short-time expansion that underlie the kernel formula.

\begin{lemma}
\label{lem:Doob-h-exact-kin}
Let $f\in C_W^4(\R^{2d})$ and define
\begin{equation}\label{eq:Dpsi-def-kin}
D_\psi f:=\frac{D_0(\psi f)}{\psi}-f\,\frac{D_0\psi}{\psi}.
\end{equation}
Then $D_\psi$ is well defined on $C_W^2(\R^{2d})$ and satisfies
\begin{equation}\label{eq:Dpsi-L-identity-kin}
D_\psi f=D_0f+2(\gamma\beta^{-1}\nabla_v\log\psi)\cdot\nabla_v f=D_0f-U\cdot\nabla_v f=\mathcal Lf.
\end{equation}
Moreover, for \(0<h\le h_0\), the following normalized ratio approximates $f + hD_{\psi}f$
\begin{equation}
\label{eq:approx_Dpsi-kin}
\frac{G_h(\psi f)}{G_h\psi}=f+h D_{\psi} f+R(h,x,v)= f+h\,\mathcal L f+R(h,x,v), 
\end{equation}
with \(\|R(h,\cdot)\|_{L_{W_{+,2m}}^\infty}\le C h^2\|f\|_{C_W^4}\), where
\[
W_{+,2m}(x,v):=(1+|x|+|v|)^{2m}W(x,v)^2W_+(x,v).
\]
\end{lemma}

\begin{proof}
Since $\psi>0$, the expression \eqref{eq:Dpsi-def-kin} is pointwise well defined. Expanding $D_0(\psi f)$ using the product rule and the fact that diffusion acts only in the $v$-variables, we obtain
\[
D_0(\psi f)=v\cdot\nabla_x(\psi f)+\gamma\beta^{-1}\Delta_v(\psi f)=\psi\,D_0f+f\,D_0\psi+2\gamma\beta^{-1}\nabla_v\psi\cdot\nabla_v f.
\]
Dividing by $\psi$ and subtracting $f\,D_0\psi/\psi$ gives
\[
D_\psi f=D_0f+2\gamma\beta^{-1}\nabla_v(\log\psi)\cdot\nabla_v f=D_0f+2(\gamma\beta^{-1}\nabla_v\log\psi)\cdot\nabla_v f.
\]
Since $\psi=e^{-\phi_{\mathrm{ki}}^\lambda/2}$,
\[
2\gamma\beta^{-1}\nabla_v\log\psi=-\gamma\beta^{-1}\nabla_v\phi_{\mathrm{ki}}^\lambda=-U,
\]
which proves \eqref{eq:Dpsi-L-identity-kin}.

For the approximation \eqref{eq:approx_Dpsi-kin}, the weighted semigroup expansions \eqref{eq:weighted-Taylor-psif-kin}-\eqref{eq:weighted-Taylor-psi-kin} yield
\[
\frac{G_h(\psi f)}{\psi}= f+ha_{\psi}+r_1(h), \qquad \frac{G_h\psi}{\psi}= 1+hb_{\psi}+r_2(h),
\]
with $a_{\psi}:={D_0(\psi f)}/{\psi}$, $b_{\psi}:={D_0\psi}/{\psi}$,
and
\begin{equation}
\label{eq:r1r2h}
\|r_1(h)\|_{L_{W_{m}^{+}}^\infty}\le C h^2\|f\|_{C_W^4}, \qquad \|r_2(h)\|_{L_{W_{m}^{+}}^\infty}\le C h^2.
\end{equation}
Hence
\[
\frac{G_h(\psi f)}{G_h\psi}=\frac{f+ha_{\psi}+r_1(h)}{1+hb_{\psi}+r_2(h)}.
\]
A direct computation gives the exact identity
\begin{equation}
\label{eq:Ghs-errors}
R:=\frac{G_h(\psi f)}{G_h\psi}-(f+hD_{\psi}f)=\frac{r_1(h)-f\,r_2(h)-h(a_{\psi}-fb_{\psi})r_2(h)-h^2b_{\psi}(a_{\psi}-fb_{\psi})}{1+hb_{\psi}+r_2(h)}.
\end{equation}
Since
\[
\frac{1}{1+hb_{\psi}+r_2(h)}=\frac{\psi}{G_h\psi},
\]
the weighted ratio bound \eqref{eq:weighted-ratio-bound-kin} gives
\[
\frac{\psi(x,v)}{G_h\psi(x,v)}\le C\,W(x,v).
\]

Moreover, \(f\in C_W^4\) implies \(f\in L_W^\infty\). Since $a_{\psi}-fb_{\psi}=D_\psi f=\mathcal L f$, and since the coefficients of $\mathcal L$ have at most polynomial growth, there exists an integer \(m\) such that
\[
|f(x,v)|\le C W(x,v)\|f\|_{C_W^4},
\qquad
|a_{\psi}(x,v)-f(x,v)b_{\psi}(x,v)| \le C W_m(x,v)\|f\|_{C_W^4}.
\]
 
On the other hand, \(b_{\psi}=D_0\psi/\psi\) depends only on \(\psi\), and by direct differentiation of \(\psi=e^{-\phi_{\mathrm{ki}}^\lambda/2}\), $b_{\psi}$ has at most polynomial growth. After enlarging \(m\) if necessary, we have
\[
|b_{\psi}(x,v)|\le C(1+|x|+|v|)^m.
\]
Using the bounds on \(r_1(h)\) and \(r_2(h)\) in \eqref{eq:r1r2h}, we obtain
\[
|r_1(h)|\le C h^2 W_{m}^{+}\|f\|_{C_W^4}, \quad |f\,r_2(h)|\le C h^2 W\,W_{m}^{+}\|f\|_{C_W^4}.
\]
Combining the above gives
\begin{align*}
&|(a_{\psi}-fb_{\psi})r_2(h)|\le C h^2 W_mW_{m}^{+}\|f\|_{C_W^4}, \\ &|b_{\psi}(a_{\psi}-fb_{\psi})|\le C (1+|x|+|v|)^mW_m\|f\|_{C_W^4}.
\end{align*}
Since \(h\le h_0\le 1\) and \(W\le W_+\), all numerator terms in \eqref{eq:Ghs-errors} are pointwise bounded by
\[
C h^2\,\frac{W_{+,2m}(x,v)}{W(x,v)}\,\|f\|_{C_W^4}.
\]
Multiplying by \(\psi/G_h\psi\le C W\) from the denominator yields the desired estimate in  \eqref{eq:approx_Dpsi-kin}
\[
\|R(h,\cdot)\|_{L_{W_{+,2m}}^\infty}\le C h^2\|f\|_{C_W^4}.
\]
\end{proof}

Lemma~\ref{lem:Doob-h-exact-kin} shows that the ratio \(G_h(\psi f)/(G_h\psi)\) gives a first-order approximation of the density evolution generated by \(\mathcal L\). This suggests replacing the infinitesimal action of $\mathcal L$ by the explicit ratio $G_h(\psi f)/(G_h\psi)$, which leads to a one-step kernel formula to approximate the evolution of the phase space densities. For the forward-Euler discretization of \eqref{eq:kinetic-FP} with step size $h$, we have $\rho_{k+1}=\rho_k+h\mathcal L^*\rho_k$. Then by duality,
\begin{align*}
&\int_{\R^{2d}} f\,\rho_{k+1}\,dx\,dv-\int_{\R^{2d}} f\,\rho_k\,dx\,dv  = \int_{\R^{2d}}\left[\frac{G_h(\psi f)}{G_h\psi}-f\right]\rho_k\,dx\,dv+O(h^2) \\
=&\int_{\R^{2d}}f(x,v)\psi(x,v)\int_{\R^{2d}}G_h(x,v\mid y,w)\frac{\rho_k(y,w)}{(G_h\psi)(y,w)}\,dy\,dw\,dx\,dv\\&- \int_{\R^{2d}}f(x,v)\rho_k(x,v)\,dx\,dv+O(h^2).
\end{align*}

Neglecting the local \(O(h^2)\) term in this weak identity motivates the following definition.

\begin{definition}[One-step kinetic kernel operator]
\label{def:kernel-operator-kin}
Let $h_0$ be as in Lemma~\ref{lem:weighted-ratio-bound-kin}. For $0<h\le h_0$, $\rho\in L^1(\R^{2d})$, $\rho\ge0$, $G_h$ in \eqref{eq:G0-kin}, and $\psi$ in \eqref{eq:psi-kin}, the kernel formula for the kinetic Fokker-Planck equation is
\begin{equation}\label{eq:kernel-kin}
(\mathcal K_h\rho)(x,v):=\psi(x,v)\int_{\R^{2d}}G_h(x,v\mid y,w)\frac{\rho(y,w)}{(G_h\psi)(y,w)}\,dy\,dw.
\end{equation}
\end{definition}

Lemma~\ref{lem:weighted-ratio-bound-kin} ensures that the denominator is finite and strictly positive for $0<h\le h_0$. Since the integrand is nonnegative, Tonelli's theorem gives, for every $\rho\ge0$ in $L^1(\R^{2d})$,
\begin{align*}
&\int_{\R^{2d}}(\mathcal K_h\rho)(x,v)\,dx\,dv
=\int_{\R^{2d}}\psi(x,v)\int_{\R^{2d}}G_h(x,v\mid y,w)\frac{\rho(y,w)}{(G_h\psi)(y,w)}\,dy\,dw\,dx\,dv \\
&=\int_{\R^{2d}}\frac{\rho(y,w)}{(G_h\psi)(y,w)}\int_{\R^{2d}}G_h(x,v\mid y,w)\psi(x,v)\,dx\,dv\,dy\,dw=\int_{\R^{2d}}\rho(y,w)\,dy\,dw.
\end{align*}
Because the right-hand side is finite, $\mathcal K_h\rho$ is finite almost everywhere, is nonnegative, belongs to $L^1(\R^{2d})$, and preserves mass. More generally, the same formula may be used for $h>h_0$ whenever $0<(G_h\psi)(y,w)<\infty$ for every $(y,w)$. For every bounded measurable $f$, absolute integrability and Fubini's theorem give the dual formula
\begin{equation}\label{eq:dual-kernel-kin}
\int_{\R^{2d}} f\,(\mathcal K_h\rho)\,dx\,dv=\int_{\R^{2d}}\frac{G_h(\psi f)(y,w)}{G_h\psi(y,w)}\,\rho(y,w)\,dy\,dw.
\end{equation}
For an unbounded measurable test function \(f\), the same identity follows by truncation whenever
\[
\int_{\R^{2d}}\frac{G_h(\psi|f|)}{G_h\psi}\rho<\infty,
\]
or equivalently whenever \(\int_{\R^{2d}}|f|\,\mathcal K_h\rho<\infty\). The weighted semigroup and moment bounds used below ensure this condition for the test functions appearing in the consistency proof.

Weak consistency of the kernel formula follows directly from Lemma~\ref{lem:Doob-h-exact-kin} as in the following proposition.  

\begin{proposition}[Weak consistency]
\label{prop:kernel-kin}
For every $T>0$, there exist $h_T\in(0,\min\{h_0,T\}]$ and a constant $C_T>0$ such that for every $f\in C_W^4(\R^{2d})$, every probability density $\rho$ satisfying $m_{W_{+,2m}}(\rho)<\infty$, and every $h\in(0,h_T]$,
\begin{equation}\label{eq:weak-error-kin}
\left|\int_{\R^{2d}}\left[f\,(\mathcal K_h\rho)-f\,\rho-h(\mathcal L f)\rho\right]\,dx\,dv\right|\le C_T h^2\|f\|_{C_W^4}\,m_{W_{+,2m}}(\rho).
\end{equation}
In particular, on sets of probability densities with uniformly bounded \(W_{+,2m}\)-moment, \(\mathcal K_h\) is weakly consistent with \(\mathcal L^*\), with one-step consistency residual of order \(h^2\).
\end{proposition}

\begin{proof}
The weighted semigroup estimate in Lemma~\ref{lem:Doob-h-exact-kin} and the assumed moment bound make the terms below absolutely integrable, so the dual formula \eqref{eq:dual-kernel-kin} applies by truncation:
\[
\int_{\R^{2d}} f\,(\mathcal K_h\rho)\,dx\,dv-\int_{\R^{2d}} f\,\rho\,dx\,dv=\int_{\R^{2d}}\left[\frac{G_h(\psi f)}{G_h\psi}-f\right]\rho\,dx\,dv.
\]
By Lemma~\ref{lem:Doob-h-exact-kin},
\[
\frac{G_h(\psi f)}{G_h\psi}=f+h\mathcal L f+R(h,\cdot), \qquad \|R(h,\cdot)\|_{L_{W_{+,2m}}^\infty}\le C_T h^2\|f\|_{C_W^4}.
\]
Substituting this into the previous identity and using
\[
\int_{\R^{2d}}|R(h,z)|\rho(z)\,dz\le \|R(h,\cdot)\|_{L_{W_{+,2m}}^\infty}\,m_{W_{+,2m}}(\rho),
\]
we obtain \eqref{eq:weak-error-kin}.
\end{proof}

Next, we state the backward regularity and moment hypotheses needed for the global error estimate. These hypotheses are not consequences of Assumption~\ref{ass:kinetic-structure} established in this paper. Conditional on them, one-step weak consistency implies finite-time weak convergence of the iterated exact-normalization kernel scheme to the solution of the kinetic Fokker-Planck equation.
\begin{theorem}[Conditional weak convergence for exact normalization]
\label{thm:global-weak-kin}
Let $\rho_0$ be a probability density satisfying $m_{W_{+,2m}}(\rho_0)<\infty$ and fix $T>0$. For each $0<h\le h_0$, set $\rho_0^h:=\rho_0$ and define $\rho_{k+1}^h:=\mathcal K_h\rho_k^h$ for $k\ge0$. Suppose there exists a solution $\rho(t,\cdot)$ for the kinetic Fokker-Planck equation~\eqref{eq:kinetic-FP} with initial data $\rho_0$. Assume that for every $f\in C_W^4(\R^{2d})$, the backward equation
\[
\partial_t u=\mathcal L u, \qquad u(0,\cdot)=f,
\]
admits a classical solution on $[0,T]\times\R^{2d}$ satisfying
\begin{equation}\label{eq:backward regularity-global-kin}
\sup_{0\le t\le T}\|u(t,\cdot)\|_{C_W^4}+\sup_{0\le t\le T}\|\partial_{tt}u(t,\cdot)\|_{L_{W_{+,2m}}^\infty}\le C_T\|f\|_{C_W^4}.
\end{equation}
Assume also that the forward and backward solutions satisfy the duality relation
\begin{equation}\label{eq:backward-forward-duality-kin}
\int_{\R^{2d}} f(z)\rho(t,z)\,dz=\int_{\R^{2d}}u(t,z)\rho_0(z)\,dz,\qquad 0\le t\le T.
\end{equation}
Assume moreover that there exist $h_T\in(0,\min\{h_0,1\}]$ and $M_T>0$ such that
\begin{equation}\label{eq:weighted-moment-bound-global-kin}
\sup_{0\le t\le T}m_{W_{+,2m}}(\rho(t,\cdot))+\sup_{h\in(0,h_T]}\sup_{k:\,kh\le T}m_{W_{+,2m}}(\rho_k^h)\le M_T.
\end{equation}
The integrals in \eqref{eq:backward-forward-duality-kin} are absolutely convergent by \eqref{eq:backward regularity-global-kin} and \eqref{eq:weighted-moment-bound-global-kin}.
Then there exists a constant $C_T>0$ such that, for every $f\in C_W^4(\R^{2d})$, every $h\in(0,h_T]$, and every integer $n\ge0$ with $nh\le T$,
\begin{equation}\label{eq:global-weak-error-kin}
\left|\int_{\R^{2d}} f(z)\rho_n^h(z)\,dz-\int_{\R^{2d}} f(z)\rho(nh,z)\,dz\right|\le C_T h\|f\|_{C_W^4}.
\end{equation}
\end{theorem}

\begin{proof}
Let $u$ solve the backward equation $\partial_tu=\mathcal Lu$ with $u(0,\cdot)=f$, and define
\[
u_k(z):=u((n-k)h,z), \qquad k=0,\dots,n,
\]
so that $u_0(z)=u(nh,z)$ and $u_n=f$. By \eqref{eq:backward-forward-duality-kin},
\[
\int_{\R^{2d}} f(z)\rho(nh,z)\,dz=\int_{\R^{2d}} u(nh,z)\rho_0(z)\,dz=\int_{\R^{2d}} u_0(z)\rho_0(z)\,dz.
\]
Hence
\begin{equation}\label{eq:telescope-kin}
\int_{\R^{2d}} f\,\rho_n^h\,dz-\int_{\R^{2d}} f\,\rho(nh)\,dz=\int_{\R^{2d}} u_n\,\rho_n^h\,dz-\int_{\R^{2d}} u_0\,\rho_0\,dz=\sum_{k=0}^{n-1}(E_{k+1}-E_k),
\end{equation}
where
\[
E_k:=\int_{\R^{2d}}u_k(z)\rho_k^h(z)\,dz.
\]
We split each increment as $E_{k+1}-E_k=A_k+B_k$, where
\[
A_k:=\int_{\R^{2d}} u_{k+1}\rho_{k+1}^h\,dz-\int_{\R^{2d}} u_{k+1}\rho_k^h\,dz-h\int_{\R^{2d}} (\mathcal Lu_{k+1})\rho_k^h\,dz,
\]
and
\[
B_k:=\int_{\R^{2d}}\bigl(u_{k+1}+h\mathcal Lu_{k+1}-u_k\bigr)\rho_k^h\,dz.
\]

Since $\rho_{k+1}^h=\mathcal K_h\rho_k^h$, Proposition~\ref{prop:kernel-kin} applied to the test function $u_{k+1}$ gives
\[
|A_k|\le C_T h^2 \|u_{k+1}\|_{C_W^4}\,m_{W_{+,2m}}(\rho_k^h)\le C_T h^2 \|f\|_{C_W^4},
\]
where we used \eqref{eq:backward regularity-global-kin} and \eqref{eq:weighted-moment-bound-global-kin}.

Next, since $\partial_tu=\mathcal Lu$, Taylor expansion in time gives
\[
u_k=u_{k+1}+h\mathcal Lu_{k+1}+r_k,
\]
with
\[
\|r_k\|_{L_{W_{+,2m}}^\infty}\le \frac{h^2}{2}\sup_{t\in[(n-k-1)h,(n-k)h]}\|\partial_{tt}u(t,\cdot)\|_{L_{W_{+,2m}}^\infty}.
\]
Therefore
\[
u_{k+1}+h\mathcal Lu_{k+1}-u_k=-r_k,
\]
and hence
\[
|B_k|\le \|r_k\|_{L_{W_{+,2m}}^\infty}\,m_{W_{+,2m}}(\rho_k^h)\le C_T h^2\|f\|_{C_W^4},
\]
again by \eqref{eq:backward regularity-global-kin} and \eqref{eq:weighted-moment-bound-global-kin}.

Summing over $k=0,\dots,n-1$ in \eqref{eq:telescope-kin} and using $nh\le T$, we obtain
\[
\left|\int_{\R^{2d}} f\,\rho_n^h\,dz-\int_{\R^{2d}} f\,\rho(nh)\,dz\right|\le \sum_{k=0}^{n-1}(|A_k|+|B_k|)
\le 2C_TT\,h\,\|f\|_{C_W^4}.
\]
Absorbing the factor $2T$ into the constant proves \eqref{eq:global-weak-error-kin}.
\end{proof}

We remark that since the uniform discrete weighted-moment bound in \eqref{eq:weighted-moment-bound-global-kin} is assumed rather than proved, Theorem~\ref{thm:global-weak-kin} is conditional and it applies only to the exact operator with denominator \(G_h\psi\).

\subsection{Kernel formula from optimal control}
\label{subsec:kernel-OC}

The previous subsection shows that the kernel operator is a weakly accurate one-step approximation of the kinetic Fokker-Planck equation. We now give a formal variational derivation of the same kernel formula \eqref{eq:kernel-kin} through a quadratic optimal control problem and the Hopf-Cole transformation.
We emphasize that related variational formulations have been studied extensively for kinetic and non-gradient Fokker-Planck equations \cite{huang2000variational,duong2014conservative,duong2022entropic}. Their numerical realizations often lead to implicit schemes or auxiliary variational subproblems. Here, our purpose is more specific: we record an exact entropy identity for a class of smooth admissible pairs and show that any sufficiently regular solution of the resulting conditional optimality system has terminal density represented by the kernel formula. We do not prove existence of a minimizer or a Lagrange multiplier.

For the remainder of this subsection only, we impose the additional integrability assumption
\[
Z_\lambda:=\int_{\R^{2d}}e^{-\phi_{\mathrm{ki}}^\lambda(x,v)}\,dx\,dv<\infty, \qquad \rho_\phi(x,v):=Z_\lambda^{-1}e^{-\phi_{\mathrm{ki}}^\lambda(x,v)}.
\]
Indeed, completing the square gives
\[
Z_\lambda=\left(\frac{2\pi}{\beta}\right)^{d/2}\int_{\R^d}\exp\!\left(-\lambda V(x)+\frac{\beta}{2\gamma^2}|\nabla_xV(x)|^2\right)\,dx.
\]
Thus, this condition is additional to Assumption~\ref{ass:kinetic-structure} and can fail for superquadratic potentials including some of the sampling examples below. It is used only to define the terminal relative-entropy penalty in the variational interpretation; it is not required for the kernel construction, the weak-consistency and convergence results above, or the sampling experiments in Section~\ref{subsec:kinetic-sampling}.

The next proposition formulates the control problem, gives an exact entropy-transport identity for classically admissible pairs, and records the first-order optimality system under an explicit multiplier hypothesis.
 
\begin{proposition}[Control identity and conditional optimality system]
\label{prop:OC-kin}
Fix $0<h\le h_0$. Call a pair $(\rho,u)$ classically admissible if $\rho>0$, $\rho\in C^1([0,h];C^0(\R^{2d}))\cap C([0,h];C^2(\R^{2d}))$, $u\in C([0,h];C^1(\R^{2d};\R^d))$, $\rho_t\in\mathcal P_W(\R^{2d})$ for all $t\in[0,h]$, and, with $\tilde u:=u+(\gamma/\beta)\nabla_v\log\rho$, the following bounds hold:
\[
\sup_{t\in[0,h]}\int_{\R^{2d}}\rho_t\bigl(1+|x|^2+|v|^2+|\phi_{\mathrm{ki}}^\lambda|+|\log\rho_t|\bigr)\,dx\,dv<\infty,
\]
\[
\int_0^h\!\int_{\R^{2d}}\rho\bigl(|u|^2+|\tilde u|^2+|\nabla_v\log\rho|^2\bigr)\,dx\,dv\,dt<\infty,
\]
\[
\int_0^h\!\int_{\R^{2d}}\left[|\partial_t\rho|(1+|\log\rho|)+\rho(1+|\log\rho|)(|v|+|\tilde u|)+\rho|\tilde u||\nabla_v\log\rho|\right]dx\,dv\,dt<\infty.
\]
Assume in addition that the spatial fluxes decay sufficiently fast for the cutoff boundary terms in the integrations by parts below to vanish. Then the following statements hold.

\begin{enumerate}[label=\textup{(\arabic*)}]
\item \textbf{Quadratic optimal control formulation.}
Consider the problem of minimizing
\begin{equation}\label{eq:J-kin}
J(\rho,u)
:= \frac{\beta}{2\gamma}\int_0^h\!\int_{\R^{2d}}   \|u\|^2\,\rho\,dx\,dv\,dt + \int_{\R^{2d}}\phi_{\mathrm{ki}}^\lambda(x,v)\rho_h(x,v)\,dx\,dv
\end{equation}
over classically admissible pairs $(\rho,u)$, subject to
\begin{equation}\label{eq:state-kin}
\partial_t\rho=-v\cdot\nabla_x\rho+\gamma\beta^{-1}\Delta_v\rho+\nabla_v\cdot(\rho u), \qquad \rho\big|_{t=0}=\rho_0.
\end{equation}

\item \textbf{Equivalent entropy-transport formulation.}
Define
\[
H(\rho):=\int_{\R^{2d}}\rho\log\rho\,dx\,dv,\qquad \mathcal I(\rho):=\int_{\R^{2d}}\|\nabla_v\log\rho\|^2\rho\,dx\,dv.
\]
Then the state equation \eqref{eq:state-kin} becomes
\begin{equation}\label{eq:w-state-kin}
\partial_t\rho=-v\cdot\nabla_x\rho+\nabla_v\cdot(\rho\tilde u),\qquad\rho\big|_{t=0}=\rho_0,
\end{equation}
and the functional $J(\rho,u)$ satisfies the exact identity
\begin{equation}\label{eq:J-entropy-FI-kin}
\widetilde J(\rho,\tilde u):=\frac{\beta}{2\gamma}\int_0^h\!\int_{\R^{2d}} \| \tilde u\|^2\rho\,dx\,dv\,dt+\frac{\gamma}{2\beta}\int_0^h\mathcal I (\rho)\,dt+\KL(\rho_h\|\rho_\phi).
\end{equation}
\begin{equation}\label{eq:J-exact-identity-kin}
J(\rho,u)=\widetilde J(\rho,\tilde u)-H(\rho_0)-\log Z_\lambda.
\end{equation}

\item \textbf{Conditional first-order optimality system.}
If a classically admissible local minimizer of \textup{(1)} admits a sufficiently regular Lagrange multiplier $\Phi$ for the state constraint, and if the first variations and integrations by parts in the multiplier argument are valid, then stationarity under admissible smooth variations with compact spatial support, including terminal density variations tangent to the probability constraint, gives
\begin{equation}\label{eq:coupled-system-kin}
\begin{cases}
\partial_t\rho = D_0^*\rho+\nabla_v\cdot(\rho u),\\[4pt]
-\partial_t\Phi = D_0\Phi+\frac{\gamma}{2\beta}\| \nabla_v\Phi\|^2,\\[4pt]
u=-\frac{\gamma}{\beta}\nabla_v\Phi,\qquad\rho\big|_{t=0}=\rho_0,\qquad\Phi\big|_{t=h}=-\phi_{\mathrm{ki}}^\lambda.
\end{cases}
\end{equation}
If terminal variations are restricted to probability densities, the terminal condition is determined up to an additive constant, $\Phi(h,\cdot)=-\phi_{\mathrm{ki}}^\lambda+C$.
\end{enumerate}

In particular, \textup{(1)} and \textup{(2)} describe the same variational problem up to the explicit constant in \eqref{eq:J-exact-identity-kin}, while \textup{(3)} is conditional on the existence and regularity of a multiplier. No existence or uniqueness assertion for a minimizer or multiplier is made.
\end{proposition}

\begin{proof}
\emph{Entropy identity in \textup{(2)}.}
Writing $u=\tilde u-\frac{\gamma}{\beta}\nabla_v\log\rho$ in \eqref{eq:state-kin} and using $\nabla_v\cdot(\rho \nabla_v\log\rho)  =\Delta_v\rho$ gives \eqref{eq:w-state-kin}. For the cost, the same relation yields
\[
\frac{\beta}{\gamma}\|u\|^2=\frac{\beta}{\gamma}\|\tilde u\|^2-2\tilde u\cdot\nabla_v\log\rho+\frac{\gamma}{\beta}\| \nabla_v\log\rho\|^2,
\]
which implies
\begin{align}
\frac{\beta}{\gamma}\int_0^h\!\!\int_{\R^{2d}}\|u\|^2\rho\,dx\,dv\,dt={}&\int_0^h\!\!\int_{\R^{2d}}\left(\frac{\beta}{\gamma}\|\tilde u\|^2\rho -2 \tilde u\cdot\nabla_v\rho\right)\,dx\,dv\,dt+\frac{\gamma}{\beta}\int_0^h\mathcal I(\rho_t)\,dt .
\label{eq:quad-decomp-kin}
\end{align}

The admissibility hypotheses justify the entropy chain rule and cutoff integrations by parts. Since the transport field $(v,0)$ is divergence free, they give
\begin{equation}\label{eq:entropy-chain-kin}
\frac{d}{dt}H(\rho_t)=-\int_{\R^{2d}}\tilde u\cdot\nabla_v\rho\,dx\,dv.
\end{equation}

Integrating in time and substituting into \eqref{eq:quad-decomp-kin}, we obtain
\begin{align*}
J(\rho,u)
&=\frac{\beta}{2\gamma}\int_0^h\!\!\int_{\R^{2d}}\|\tilde u\|^2\rho\,dx\,dv\,dt+ \frac{\gamma}{2\beta}\int_0^h\mathcal I(\rho)\,dt \\
&\quad+\int_{\R^{2d}}\rho_h\log\rho_h\,dx\,dv-\int_{\R^{2d}}\rho_0\log\rho_0\,dx\,dv+\int_{\R^{2d}}\phi_{\mathrm{ki}}^\lambda\rho_h\,dx\,dv \\
&=\widetilde J(\rho,\tilde u)-H(\rho_0)-\log Z_\lambda.
\end{align*}
This proves \eqref{eq:J-exact-identity-kin}.

\emph{Derivation of \textup{(3)} from \textup{(1)}.}
Under the multiplier and regularity hypotheses in \textup{(3)}, the first-order conditions for \eqref{eq:J-kin}-\eqref{eq:state-kin} introduce a Lagrange multiplier $\Phi$ for the state equation. After integration by parts, the adjoint equation is
\[
-\partial_t\Phi=D_0\Phi-u\cdot\nabla_v\Phi-\frac{\beta}{2\gamma}\langle u,u\rangle,\qquad \Phi(h,\cdot)=-\phi_{\mathrm{ki}}^\lambda.
\]
The first-order condition with respect to $u$ gives the feedback law $u=-\gamma\beta^{-1}\nabla_v\Phi$. Substituting this into the adjoint equation gives the second equation in \eqref{eq:coupled-system-kin}.
\end{proof}

The next proposition shows, via a Hopf-Cole transformation, that the kinetic kernel formula represents the terminal density associated with a sufficiently smooth solution of the conditional optimality system \eqref{eq:coupled-system-kin}.
 
\begin{proposition}[Kernel representation of a classical optimality-system solution]
\label{prop:kernel-OC-kin}
Let $0<h\le h_0$ and assume that the system \eqref{eq:coupled-system-kin} admits a classical solution on $[0,h]\times\R^{2d}$ such that $\eta=e^{\Phi/2}$ and $\zeta=\rho/\eta$ belong to uniqueness classes for the semigroup equations used below, with $\psi$ and $\rho_0/(G_h\psi)$ in the corresponding semigroup domains. Then the terminal density is given by
\[
\rho_h=\mathcal K_h\rho_0,
\]
where $\mathcal K_h$ is the kernel operator in Definition~\ref{def:kernel-operator-kin}.
\end{proposition}

\begin{proof}
We solve the coupled system \eqref{eq:coupled-system-kin} explicitly via the Hopf-Cole substitution $\eta:=e^{\Phi/2}$.
Computing $\partial_t\eta$ and $D_0\eta$ in terms of $\Phi$, and using that the second-order part of $D_0$ acts only in the $v$-variables, we obtain
\[
D_0\eta=\frac12\eta\,D_0\Phi+\frac{\gamma}{4\beta}\eta\langle \nabla_v\Phi,\nabla_v\Phi\rangle.
\]
Substituting $\eta$ into the second equation in \eqref{eq:coupled-system-kin}, it follows that
\[
\partial_t\eta=\frac12\eta\,\partial_t\Phi=-\frac12\eta\left(D_0\Phi+\frac{\gamma}{2\beta}\langle \nabla_v\Phi,\nabla_v\Phi\rangle\right)=-D_0\eta.
\]
Hence, $\eta$ satisfies the backward Kolmogorov equation
\[
\partial_t\eta=-D_0\eta,
\qquad
\eta(h,\cdot)=e^{-\phi_{\mathrm{ki}}^\lambda/2}=\psi.
\]
Since $G_h$ is the transition kernel of $D_0$, it implies
\begin{equation}\label{eq:eta-solution-kin}
\eta(t,\cdot)=G_{h-t}\psi,
\qquad
\eta(0,\cdot)=G_h\psi.
\end{equation}

Next, set $\zeta:=\rho/\eta$. Differentiating $\rho=\eta\zeta$ and substituting the first equation in \eqref{eq:coupled-system-kin} give
\[
\eta\,\partial_t\zeta=\partial_t\rho-\zeta\,\partial_t\eta=D_0^*(\eta\zeta)-\frac{\gamma}{\beta}\nabla_v\cdot(\nabla_v\Phi\,\eta\zeta)+\zeta\,D_0\eta.
\]
Since $\nabla_v\eta=\tfrac12\eta\nabla_v\Phi$, a direct computation using the product rule shows that the right-hand side simplifies to $\eta\,D_0^*\zeta$.
Therefore
\[
\partial_t\zeta=D_0^*\zeta,
\qquad
\zeta(0,\cdot)=\frac{\rho_0}{G_h\psi}.
\]
Hence
\[
\zeta(h,x,v)=\int_{\R^{2d}}G_h(x,v\mid y,w)\frac{\rho_0(y,w)}{G_h\psi(y,w)}\,dy\,dw.
\]
Combining this with $\rho_h=\eta(h,\cdot)\zeta(h,\cdot)=\psi\,\zeta(h,\cdot)$ gives
\[
\rho_h(x,v)=\psi(x,v)\int_{\R^{2d}}G_h(x,v\mid y,w)\frac{\rho_0(y,w)}{G_h\psi(y,w)}\,dy\,dw=(\mathcal K_h\rho_0)(x,v),
\]
as claimed.
\end{proof}

Proposition~\ref{prop:OC-kin} gives the exact relation between formulations \textup{(1)} and \textup{(2)} and records a conditional first-order optimality system, while Proposition~\ref{prop:kernel-OC-kin} shows that the kinetic kernel formula represents any classical solution of that system in the stated uniqueness classes.
Thus, conditional on the existence of such a minimizer, multiplier, and classical solution, the optimal control system and the forward-backward factorization provide a formal variational characterization of the one-step kernel operator \eqref{eq:kernel-kin}.

In particular, formulation \textup{(2)} in Proposition~\ref{prop:OC-kin} suggests a Fisher-information-regularized proximal-type step. Motivated by the dynamic formulation of Wasserstein transport \cite{benamou2000computational}, for each \(0<h\leq h_0\), we define the regularized kinetic action cost
\begin{align*}
\mathsf C_{\mathrm{ki},h}(\rho_k,\rho_h):=\inf_{\rho,\tilde u}\bigg\{&\int_0^h\!\int_{\R^{2d}} |\tilde u|^2\rho\,dx\,dv\,dt+\frac{\gamma^2}{\beta^2}\int_0^h \mathcal I(\rho_t)\,dt:\\
&\partial_t\rho=-v\cdot\nabla_x\rho+\nabla_v\cdot(\rho\tilde u), \quad \rho(0,\cdot)=\rho_k, \quad \rho(h,\cdot)=\rho_h \bigg\}.
\end{align*}
Here, the infimum is over pairs \((\rho,\tilde u)\) satisfying the displayed constraints for which \(u:=\tilde u-(\gamma/\beta)\nabla_v\log\rho\) makes \((\rho,u)\) classically admissible in the sense of Proposition~\ref{prop:OC-kin}; its value is \(+\infty\) if no such pair exists. The quantity \(\mathsf C_{\mathrm{ki},h}\) depends on \(h\) and is used only as a regularized action cost; no symmetry, triangle inequality, or metric property is asserted.
Thus, the one-step update with the kernel formula  \(\mathcal K_h\rho_k\)  in \eqref{eq:kernel-kin} is formally associated with a Fisher-information-regularized kinetic transport problem with terminal penalty $\KL(\rho_h\|\rho_\phi)$
\[
\inf_{\rho_h\in\mathcal P_W(\R^{2d})} \left\{\KL(\rho_h\|\rho_\phi)+\frac{\beta}{2\gamma}\mathsf C_{\mathrm{ki},h}(\rho_k,\rho_h)\right\}.
\]
The terminal relative-entropy term depends on the choice of $\phi_{\mathrm{ki}}^\lambda$ and should therefore be viewed as an auxiliary variational penalty associated with the construction of the kernel formula, rather than as the physical free energy relative to the Gibbs equilibrium.

\section{Generalizations to Degenerate and Nondegenerate Fokker-Planck Equations}
\label{sec:generalizations}

The construction of the kernel formula in Section~\ref{sec:kernel_kinetic} extends to affine Fokker-Planck equations with constant diffusion. We state a unified formulation that includes both degenerate phase-space equations and fully nondegenerate equations, thereby avoiding a repetition of the kernel, consistency, convergence, and variational arguments.

Let \(z=(x,v)\in\R^{d_x}\times\R^{d_v}\), set \(D:=d_x+d_v\), and allow \(d_x=0\), in which case \(z=v\in\R^{d_v}\). Let \(A_v\in\R^{d_v\times d_v}\) be constant, symmetric, and positive definite, let \(b(z)=Bz+c\), and assume that \(U(z)=A_v\nabla_v\phi(z)\) for a sufficiently regular potential \(\phi\) such that the normalization and weighted estimates stated below hold. We consider
\begin{equation}\label{eq:FP-deg-general}
\partial_t\rho=-\nabla_z\cdot(b\rho)+\nabla_v\cdot(A_v\nabla_v\rho)+\nabla_v\cdot(U\rho),\qquad z\in\R^D.
\end{equation}
Define
\begin{equation}\label{eq:D0-deg-general}
D_0f:=b\cdot\nabla_zf+\nabla_v\cdot(A_v\nabla_vf),\qquad D_0^*\rho:=-\nabla_z\cdot(b\rho)+\nabla_v\cdot(A_v\nabla_v\rho),
\end{equation}
and
\begin{equation}\label{eq:L-deg-general}
\mathcal Lf:=D_0f-U\cdot\nabla_vf,\qquad \mathcal L^*\rho:=D_0^*\rho+\nabla_v\cdot(U\rho).
\end{equation}
Thus, \eqref{eq:FP-deg-general} is \(\partial_t\rho=\mathcal L^*\rho\).

To describe the reference semigroup, set  
\begin{equation}\label{eq:affine-covariance-general}
\mathcal A:=\operatorname{diag}(0,A_v),\qquad m_h:=\int_0^h e^{(h-s)B}c\,ds,\qquad Q_h:=2\int_0^h e^{(h-s)B}\mathcal A e^{(h-s)B^\top}\,ds.
\end{equation}
The zero block is omitted when \(d_x=0\). Assume that \(Q_h\) is positive definite for \(h>0\), equivalently that the Kalman rank condition
\begin{equation}\label{eq:kalman-rank-general}
\operatorname{rank}\!\left[\mathcal A^{1/2},B\mathcal A^{1/2},\ldots,B^{D-1}\mathcal A^{1/2}\right]=D
\end{equation}
holds. Then \(D_0\) generates the Gaussian Markov semigroup with transition density
\begin{equation}\label{eq:kernel-D0-affine-general}
G_h(z\mid y)=\frac{1}{(2\pi)^{D/2}(\det Q_h)^{1/2}}\exp\!\left(-\frac12\left\|z-e^{hB}y-m_h\right\|_{Q_h^{-1}}^2\right),\quad
\|q\|_{Q_h^{-1}}^2:=q^\top Q_h^{-1}q.
\end{equation}
We use the convention \((G_hg)(y):=\int_{\R^D}G_h(z\mid y)g(z)\,dz.\) Set \(\psi:=e^{-\phi/2}\), and use the analogues on \(\R^D\) of the weighted spaces and the weights \(W\), \(W_m^+\), and \(W_{+,2m}\) from Section~\ref{sec:kernel_kinetic}.  
We assume that there exist \(h_0>0\), \(C>0\), and an integer \(m\ge0\) such that, for every \(0<h\le h_0\),
\begin{equation}\label{eq:affine-weighted-assumptions}
\begin{aligned}
&0<(G_h\psi)(y)<\infty\quad\text{for every }y\in\R^D,\qquad \left\|\frac{G_h\psi}{\psi}\right\|_{L_W^\infty}+\left\|\frac{\psi}{G_h\psi}\right\|_{L_W^\infty}\le C,\\
&\left\|\frac{G_h(\psi f)}{G_h\psi}-f-h\mathcal Lf\right\|_{L_{W_{+,2m}}^\infty}\le Ch^2\|f\|_{C_W^4},\qquad f\in C_W^4(\R^D). 
\end{aligned}
\end{equation}
The second line is the abstract normalized short-time estimate proved in the kinetic setting by combining the weighted Taylor estimates with the corresponding coefficient-growth bounds. Under them, define
\begin{equation}\label{eq:kernel-nd-general}
(\mathcal K_h^{\mathrm{aff}}\rho)(z):=\psi(z)\int_{\R^D}G_h(z\mid y)\frac{\rho(y)}{(G_h\psi)(y)}\,dy.
\end{equation}
For every nonnegative \(\rho\in L^1(\R^D)\), this operator is nonnegative and preserves mass. Moreover, for every bounded measurable \(f\),
\begin{equation}\label{eq:dual-kernel-affine-general}
\int_{\R^D}f\,\mathcal K_h^{\mathrm{aff}}\rho\,dz=\int_{\R^D}\frac{G_h(\psi f)}{G_h\psi}\rho\,dz.
\end{equation}
For unbounded \(f\), the same identity holds by truncation whenever \(G_h(\psi|f|)/(G_h\psi)\in L^1(\rho)\). In the consistency statement below, the weighted hypotheses are understood to ensure this condition and the absolute integrability of \((\mathcal Lf)\rho\).

\begin{corollary}[Unified consistency and conditional convergence]
\label{cor:kernel-deg-general}
Under \eqref{eq:affine-weighted-assumptions}, for every \(T>0\), there exist \(h_T\in(0,\min\{h_0,T\}]\) and \(C_T>0\) such that, for every \(f\in C_W^4(\R^D)\), every probability density \(\rho\) satisfying \(m_{W_{+,2m}}(\rho)<\infty\), and every \(0<h\le h_T\),
\begin{equation}\label{eq:weak-error-affine-general}
\left|\int_{\R^D}\left[f\,\mathcal K_h^{\mathrm{aff}}\rho-f\rho-h(\mathcal Lf)\rho\right]dz\right|\le C_Th^2\|f\|_{C_W^4}m_{W_{+,2m}}(\rho).
\end{equation}
Let \(\rho(t)\) solve \(\partial_t\rho=\mathcal L^*\rho\) with \(\rho(0)=\rho_0\), and define \(\rho_0^h=\rho_0\) and \(\rho_{k+1}^h=\mathcal K_h^{\mathrm{aff}}\rho_k^h\). If the backward equation \(\partial_tu=\mathcal Lu\), \(u(0)=f\), has a classical solution satisfying the \(C_W^4\) and \(\partial_{tt}\) bounds in \eqref{eq:backward regularity-global-kin}, the forward-backward duality in \eqref{eq:backward-forward-duality-kin} holds, and the exact and discrete solutions satisfy the uniform \(W_{+,2m}\)-moment bound in \eqref{eq:weighted-moment-bound-global-kin}, with \(\R^{2d}\) replaced by \(\R^D\), then, possibly after decreasing \(h_T\), for every \(0<h\le h_T\) and every integer \(n\ge0\) satisfying \(nh\le T\),
\begin{equation}\label{eq:global-weak-error-affine-general}
\left|\int_{\R^D}f\rho_n^h\,dz-\int_{\R^D}f\rho(nh)\,dz\right|\le C_Th\|f\|_{C_W^4}.
\end{equation}
Thus, the exact-normalization scheme has a one-step weak consistency residual of order \(O(h^2)\) and, conditionally, first-order finite-time weak accuracy.
\end{corollary}
\begin{proof}
    The local estimate follows from \eqref{eq:dual-kernel-affine-general} and the second bound in \eqref{eq:affine-weighted-assumptions} by integration against \(\rho\), while the global estimate follows under the stated hypotheses by the telescoping argument in the proof of Theorem~\ref{thm:global-weak-kin}, with \(\R^{2d}\) replaced by \(\R^D\).
\end{proof}
Next, we consider the variational interpretation of the kernel formula. Assume additionally that
\begin{equation}\label{eq:Zphi-affine-general}
Z_\phi:=\int_{\R^D}e^{-\phi(z)}\,dz<\infty,\qquad \rho_\phi:=Z_\phi^{-1}e^{-\phi}.
\end{equation}
For \(0<h\le h_0\), consider the problem over the analogues of the classically admissible pairs in Proposition~\ref{prop:OC-kin}, with \(\gamma\beta^{-1}I_d\) replaced by \(A_v\) and the kinetic transport replaced by the affine drift \(b\):
\begin{equation}\label{eq:J-affine-general}
\begin{aligned}
\inf_{\rho,u}\quad &J_h(\rho,u):=\frac12\int_0^h\!\int_{\R^D}\langle A_v^{-1}u,u\rangle\rho\,dz\,dt+\int_{\R^D}\phi\rho_h\,dz,\\
\text{subject to}\quad &\partial_t\rho=D_0^*\rho+\nabla_v\cdot(\rho u),\qquad \rho|_{t=0}=\rho_0.
\end{aligned}
\end{equation}
With \(\widetilde u:=u+A_v\nabla_v\log\rho\), the constraint becomes \(\partial_t\rho=-\nabla_z\cdot(b\rho)+\nabla_v\cdot(\rho\widetilde u)\). Define
\begin{equation}\label{eq:entropy-fisher-affine-general}
H(\rho):=\int_{\R^D}\rho\log\rho\,dz,\qquad \mathcal I_{A_v}(\rho):=\int_{\R^D}\langle A_v\nabla_v\log\rho,\nabla_v\log\rho\rangle\rho\,dz.
\end{equation}
For every such admissible pair, the exact entropy identity is
\begin{equation}\label{eq:J-entropy-affine-general}
\begin{aligned}
\widetilde J_h(\rho,\widetilde u):={}&\frac12\int_0^h\!\int_{\R^D}\langle A_v^{-1}\widetilde u,\widetilde u\rangle\rho\,dz\,dt+\frac12\int_0^h\mathcal I_{A_v}(\rho_t)\,dt+\KL(\rho_h\|\rho_\phi),\\
J_h(\rho,u)={}&\widetilde J_h(\rho,\widetilde u)-H(\rho_0)-\log Z_\phi+h\tr(B).
\end{aligned}
\end{equation}
Indeed, \(\nabla_z\cdot b=\tr(B)\) gives
\begin{equation}\label{eq:entropy-chain-affine-general}
\frac{d}{dt}H(\rho_t)=-\tr(B)-\int_{\R^D}\widetilde u\cdot\nabla_v\rho\,dz.
\end{equation}
Thus, the term \(h\tr(B)\) is independent of \((\rho,u)\), and no trace-free assumption is needed. If a classically admissible local minimizer exists, admits a sufficiently regular Lagrange multiplier for which the first variations and integrations by parts are valid, and the transformed Kolmogorov equations are unique in the relevant classes, with \(\psi\) and \(\rho_0/(G_h\psi)\) in their corresponding semigroup domains, then the same first-order optimality system and Hopf-Cole factorization as in Propositions~\ref{prop:OC-kin} and~\ref{prop:kernel-OC-kin} give
\begin{equation}\label{eq:variational-kernel-affine-general}
\rho^{\mathrm{opt}}(h,\cdot)=\mathcal K_h^{\mathrm{aff}}\rho_0.
\end{equation}
This is a conditional representation of a classical optimality-system solution, not an existence or uniqueness theorem for the variational problem.

When \(d_x>0\), the framework above gives the general degenerate phase-space equation, provided the Kalman rank condition transfers the \(v\)-diffusion to all components. The kinetic equation in Section~\ref{sec:kernel_kinetic} is recovered by taking \(d_x=d_v=d\), \(B=\left(\begin{smallmatrix}0&I_d\\0&0\end{smallmatrix}\right)\), \(c=0\), \(A_v=\gamma\beta^{-1}I_d\), and \(\phi=\phi_{\mathrm{ki}}^\lambda\).

When \(d_x=0\), rename \(v\) as \(x\) and \(A_v\) as \(A\). Then \eqref{eq:FP-deg-general} becomes the fully nondegenerate equation
\begin{equation}\label{eq:FP-nd-general}
\partial_t\rho=\nabla\cdot(A\nabla\rho)-\nabla\cdot\bigl((Bx+c)\rho\bigr)+\nabla\cdot(A\nabla\phi\,\rho),\qquad x\in\R^d.
\end{equation}
Here \(Q_h\) in \eqref{eq:affine-covariance-general} is automatically positive definite, and \eqref{eq:kernel-D0-affine-general} and \eqref{eq:kernel-nd-general} give the explicit nondegenerate Gaussian kernel operator. Corollary~\ref{cor:kernel-deg-general} and the variational identity \eqref{eq:J-entropy-affine-general} therefore apply to both specializations without separate proofs.

\section{Numerical experiments}
\label{sec:ne}
This section presents numerical experiments illustrating the accuracy of the kernel operators developed in Sections~\ref{sec:kernel_kinetic} and~\ref{sec:generalizations}, as well as their application in deterministic kinetic sampling.
 
Throughout, we choose a time step \(h>0\) for which the normalization factors used in the relevant kernel formula are finite; the analysis above guarantees this for sufficiently small \(h\) under its assumptions. In the density-evolution benchmarks, \(\rho_k\approx\rho(kh,\cdot)\) is evaluated on a fixed grid, while in the sampling experiments it is represented by a particle ensemble $\{(x_j^k,v_j^k)\}_{j=1}^N$.

\subsection{Numerical implementation}
\label{subsec:implementation}

The kernel operators in \eqref{eq:kernel-kin} and \eqref{eq:kernel-nd-general} involve normalization integrals that become expensive to evaluate by tensor-product quadrature in high dimensions. For a weighted empirical approximation
\[
\rho^N=\sum_{j=1}^N\omega_j\delta_{z_j},\qquad \omega_j\geq 0,\qquad \sum_{j=1}^N\omega_j=1,
\]
where \(z_j\in\R^m\), the kernel update is
\begin{equation}
(\mathcal K_h\rho^N)(z)=\psi(z)\sum_{j=1}^N\omega_j\frac{G_h(z\mid z_j)}{Z_h(z_j)},\qquad Z_h:=G_h\psi.
\label{eq:kernel-particle-discrete}
\end{equation}
Thus, the principal computational task is the evaluation of the positive normalizers \(Z_h(z_j)\).

We illustrate one possible approximation for the nondegenerate kernel in \eqref{eq:kernel-nd-general}; the same principle applies to the other Gaussian kernels considered here. In this case,
\begin{equation}
\label{eq:denominator-laplace-form}
Z_h(x)=\frac{1}{C_h}\int_{\R^d}\exp\bigl(-\Phi_h(y;x)\bigr)\,dy,\qquad \Phi_h(y;x)=\frac12\bigl\|y-e^{hB}x-m_h\bigr\|_{Q_h^{-1}}^2+\frac12\phi(y),
\end{equation}
with \( C_h=(2\pi)^{d/2}(\det Q_h)^{1/2}\).
Suppose that \(\Phi_h(\cdot;x)\) has a unique nondegenerate global minimizer
\[
y_h^*(x)=\operatorname*{argmin}_{y\in\R^d}\left\{\frac12\bigl\|y-(e^{hB}x+m_h)\bigr\|_{Q_h^{-1}}^2+\frac12\phi(y)\right\}.
\]
This is a generalized proximal point of \(\phi/2\) in the metric induced by \(Q_h^{-1}\). The leading Laplace approximation is
\begin{equation}
\widehat Z_h^{\mathrm L}(x):=\frac{\exp\bigl(-\Phi_h(y_h^*(x);x)\bigr)}{(\det Q_h)^{1/2}\bigl(\det\nabla_y^2\Phi_h(y_h^*(x);x)\bigr)^{1/2}}.
\label{eq:laplace-approx-denominator}
\end{equation}
Equivalently,
\[
\log\widehat Z_h^{\mathrm L}(x)=-\Phi_h(y_h^*(x);x)-\frac12\log\det\nabla_y^2\Phi_h(y_h^*(x);x)-\frac12\log\det Q_h,
\]
which is preferable numerically when the kernel is sharply concentrated. If there are finitely many isolated nondegenerate global minimizers, the corresponding leading contributions are summed. Although the localization of the Gaussian kernel motivates this approximation for small \(h\), no uniform Laplace-error estimate is established here; see, for example, \cite{tibshirani2025laplace}.

More generally, replacing \(Z_h\) by any positive numerical approximation \(\widehat Z_h\) gives
\begin{align}
(\widehat{\mathcal K}_{h,\mathrm{raw}}\rho)(z)&:=\psi(z)\int_{\R^m}G_h(z\mid y)\frac{\rho(y)}{\widehat Z_h(y)}\,dy,\label{eq:approx-raw-operator}\\
M_h(\rho)&:=\int_{\R^m}(\widehat{\mathcal K}_{h,\mathrm{raw}}\rho)(z)\,dz=\int_{\R^m}\frac{Z_h(y)}{\widehat Z_h(y)}\rho(y)\,dy, \qquad
\widehat{\mathcal K}_h\rho:=\frac{\widehat{\mathcal K}_{h,\mathrm{raw}}\rho}{M_h(\rho)},\label{eq:approx-normalized-operator}
\end{align}
with \(0<M_h(\rho)<\infty\). The last normalization preserves positivity and unit mass but makes \(\widehat{\mathcal K}_h\) nonlinear in \(\rho\). Since it is independent of \(z\), it does not affect logarithmic score estimates.

Alternatively, on a grid with positive quadrature weights \(q_i\), using the same quadrature in the normalizers and the update gives
\[
Z_{h,j}^{Q}=\sum_iq_i\psi(z_i)G_h(z_i\mid z_j),\qquad \rho_i^+=\psi(z_i)\sum_jq_jG_h(z_i\mid z_j)\frac{\rho_j}{Z_{h,j}^{Q}},
\]
and hence
\[
\sum_iq_i\rho_i^+=\sum_jq_j\rho_j.
\]
Thus, this realization preserves mass exactly at the discrete level.

The convergence results above concern the exact operator \(\mathcal K_h\). Indeed, if
\[
\epsilon_h:=\sup_y\left|\frac{\widehat Z_h(y)}{Z_h(y)}-1\right|\leq\frac12,
\]
then, for every probability density \(\rho\) and bounded measurable \(f\),
\[
\left|\int_{\R^m}f\,\widehat{\mathcal K}_h\rho-\int_{\R^m}f\,\mathcal K_h\rho\right|\leq4\epsilon_h\|f\|_\infty.
\]
Thus, denominator approximation produces an \(O(\epsilon_h)\) one-step perturbation for bounded $f$. A finite-time estimate of order \(O(h+\epsilon_h/h)\) would additionally require weighted relative-error, moment, and stability estimates for the approximate iterates, which are not proved here. 

In the density-evolution experiments, the normalizers are evaluated using positive grid quadrature, after which the raw density update is formed and mass-normalized as in \eqref{eq:approx-raw-operator}-\eqref{eq:approx-normalized-operator}. In the sampling experiments, the normalizers are approximated by the Laplace procedure, but only the logarithmic score of the resulting kernel mixture is evaluated; the global mass-normalization factor is not computed because it is independent of the evaluation point and therefore cancels from the score.
\begin{algorithm}[t]
\caption{Mass-normalized time-discrete scheme with an approximate denominator}
\label{alg:kernel_time_stepping}
\begin{algorithmic}[1]
\STATE Initialize a density $\rho_0$. Choose $T>0$ and an admissible step size $h>0$, and set $n=\lfloor T/h \rfloor$.
\FOR{$k=1,\ldots,n$}
    \STATE Compute a numerical approximation $\widehat Z_h$ of $Z_h=G_h\psi$.
    \STATE Form the raw update \(\widetilde\rho_k(z)=\psi(z)\int_{\R^m}G_h(z\mid y)\frac{\rho_{k-1}(y)}{\widehat Z_h(y)}\,dy\).
    \STATE Normalize the mass by setting \(\rho_k=\frac{\widetilde\rho_k}{\int_{\R^m}\widetilde\rho_k(z)\,dz}.\)
\ENDFOR
\end{algorithmic}
\end{algorithm}
 
\subsection{Density evolution}
\label{subsec:density-evolution} 
We consider the kinetic Fokker-Planck equation~\eqref{eq:kinetic-FP} on phase space $x,v\in\R$, with \(V(x)=x^2/2\), \(\beta=1\), and \(\gamma=5\). The kernel weight uses \(\lambda=0.34\). The initial density is the nine-component Gaussian mixture
\begin{equation}
\label{eq:rho0-density-2dtoy}
\rho_0(x,v)=\sum_{\ell=1}^9p_\ell\,\mathcal N\bigl((x,v);m_\ell,\Sigma_\ell\bigr),
\end{equation}
with unequal weights and component covariances that are reported in the appendix.
Since \(V\) is quadratic, each Gaussian component remains Gaussian. The reference density and score are computed by evolving the component means and covariances using the exact Gaussian evolution formulas and then recombining the resulting mixture.

The exact kinetic equation and kernel formula are posed on \(\R^2\), whereas all numerical quadratures in this benchmark are restricted to the finite box \(x\in[-7.5,7.5]\), \(v\in[-7,7]\). On this box, the denominator and numerator are evaluated on the same phase-space grid using a hybrid midpoint-kernel discretization: a normalized convolution is used when the position Gaussian is resolved, shifted interpolation is used when it is under-resolved, and velocity is integrated by quadrature.  Consequently, this is a finite-box numerical surrogate. We use \(T=0.35\), step sizes \(h\in\{0.175,0.0875,0.05,0.025\}\), and grids \(N_x\times N_v\in\{64\times32,158\times40,480\times52,1600\times72\}\), and denote the resulting approximation by \(\widehat{\mathcal K}_h\). 

For the grid computation, the nodal masses in \eqref{eq:kernel-particle-discrete} are \(\omega_j=q_j\rho_{k,j}/\sum_{\ell=1}^Nq_\ell\rho_{k,\ell}\) with quadrature weights \(q_j\) and nodal density values \(\rho_{k,j}\). Differentiating this discrete kernel mixture analytically gives
\begin{equation}\label{eq:score-estimator-general}
\nabla_z\log(\mathcal K_h\rho_k^N)(z)=\nabla_z\log\psi(z)+\sum_{j=1}^Nw_j(z)\nabla_z\log G_h(z\mid z_j),
\end{equation}
where
\begin{equation}\label{eq:wj-kernel}
w_j(z):=\frac{\omega_jG_h(z\mid z_j)/Z_h(z_j)}{\sum_{\ell=1}^N\omega_\ell G_h(z\mid z_\ell)/Z_h(z_\ell)},\qquad \sum_{j=1}^Nw_j(z)=1.
\end{equation}
When a positive numerical denominator \(\widehat Z_h\) is used, we just replace all $Z_h$ above by  \(\widehat Z_h\). The reported score error is
\[
\frac{\|s_h-s_{\mathrm{ref}}\|_{L^2(\rho_{\mathrm{ref}})}}{\|s_{\mathrm{ref}}\|_{L^2(\rho_{\mathrm{ref}})}},\qquad s=\nabla_{(x,v)}\log\rho.
\]
The gradient $\nabla_{(x,v)}\log G_h$ is computed from the explicit form~\eqref{eq:G0-kin} as
\begin{equation}
\begin{cases}
\nabla_x\log G_h(x,v\mid x_j,v_j)=-\dfrac{6\beta}{\gamma h^3}\left(x-x_j-\dfrac h2(v+v_j)\right),\\[3pt]
\nabla_v\log G_h(x,v\mid x_j,v_j)=\dfrac{3\beta}{\gamma h^2}\left(x-x_j-\dfrac h2(v+v_j)\right)-\dfrac{\beta}{2\gamma h}(v-v_j).
\end{cases}
\end{equation}

\begin{figure}[t]
\centering
\includegraphics[width=0.225\textwidth]{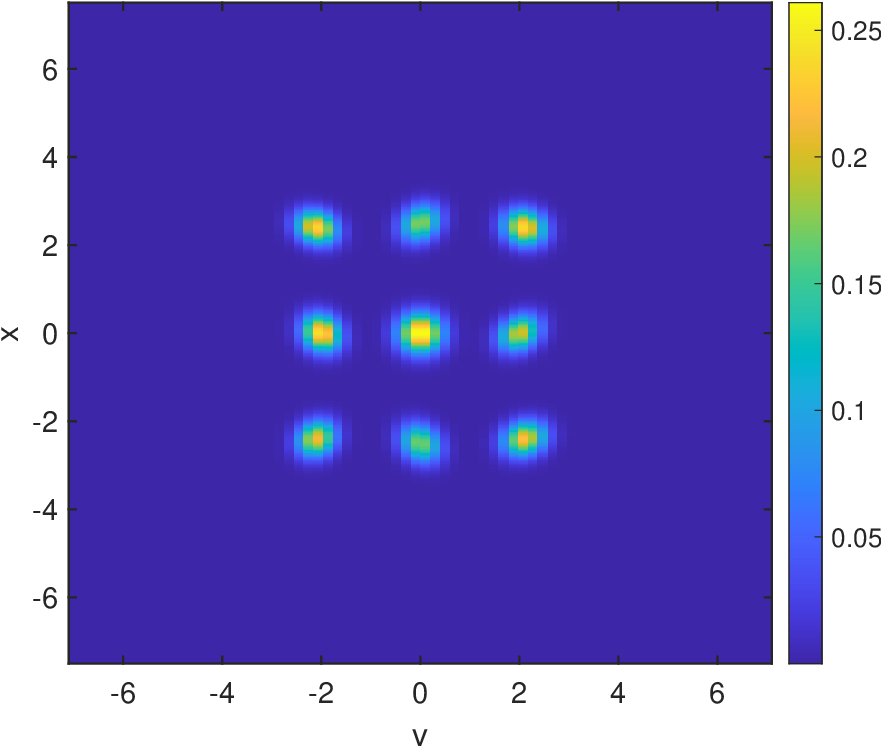}
\includegraphics[width=0.225\textwidth]{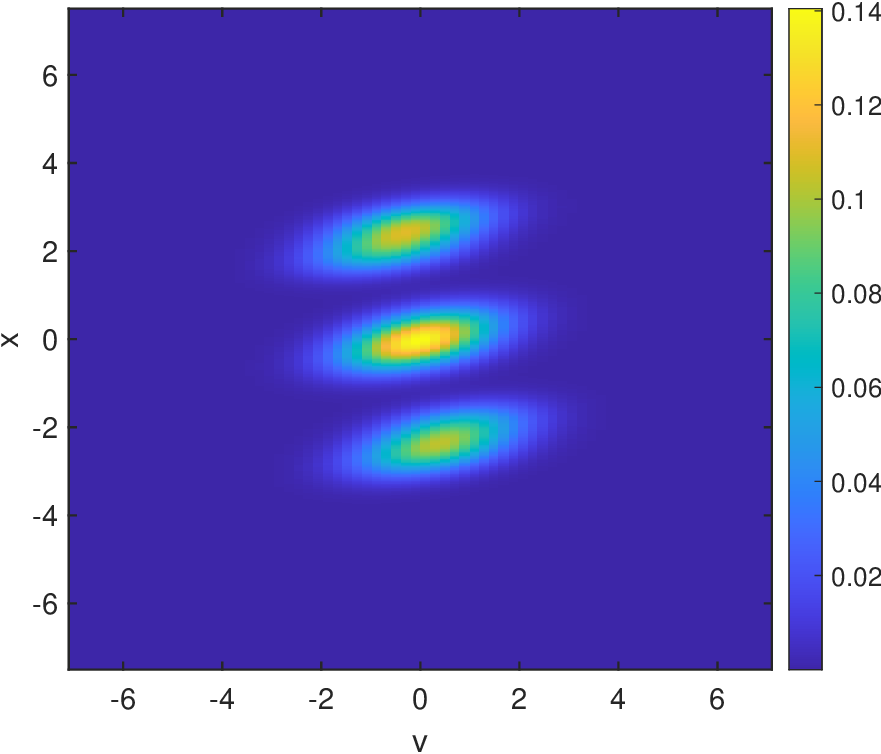}
\includegraphics[width=0.225\textwidth]{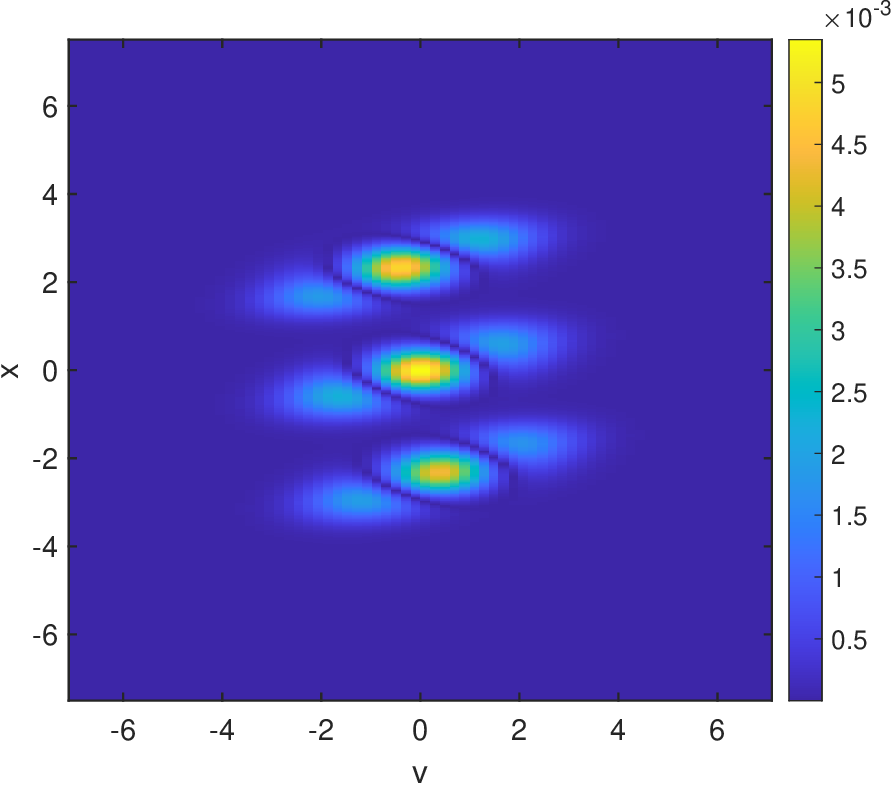}
\includegraphics[width=0.27\textwidth,trim=0cm 0cm 0cm 0.1cm,clip]{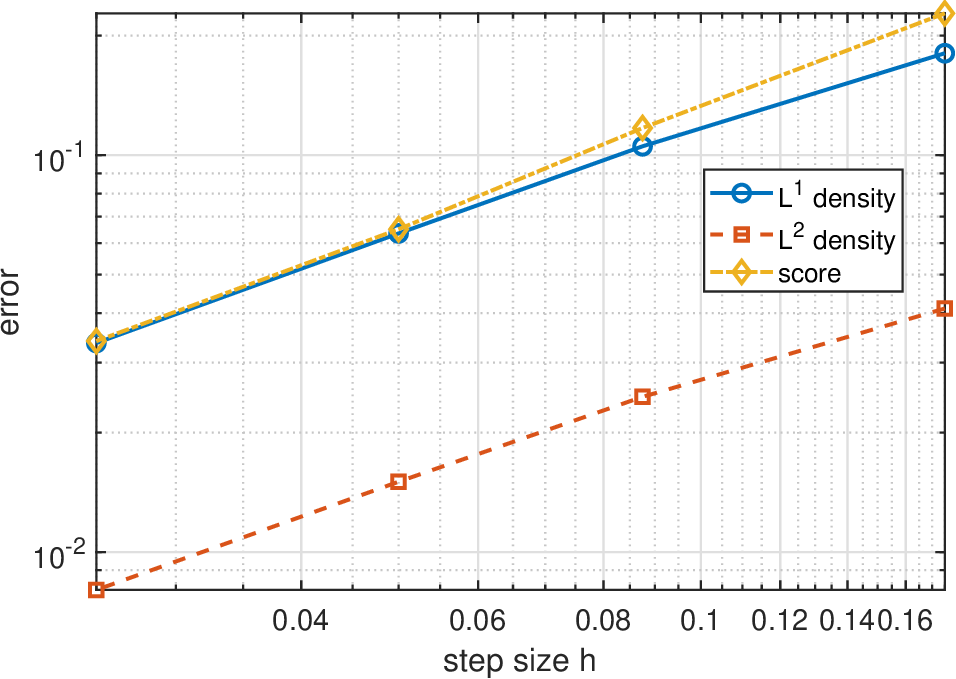}
\caption{Two-dimensional kinetic Fokker-Planck equation with Gaussian-mixture initial data. From left to right: initial density $\rho_0$; kernel approximation $\rho_n^h$ at \(T=0.35\) and \(h=0.025\); pointwise absolute error $|\rho_n^h-\rho_{\mathrm{ref}}|$; \(L^1\) density error (solid), \(L^2\) density error (dashed), and relative \(L^2(\rho_{\mathrm{ref}})\) score error (dash-dotted), as functions of \(h\) on a log-log scale.}
\label{fig_toy_2d}
\end{figure}

Figure~\ref{fig_toy_2d} shows that the kernel formula provides an accurate approximation of the density evolution and its score function $\nabla_{(x,v)}\log\rho$ in this test case.

\subsection{Kinetic sampling via the score estimator from the kernel formula}
\label{subsec:kinetic-sampling}

We now apply the kernel framework to the problem of sampling from a target distribution $\rho^*\propto e^{-\beta V}$ on $\R^d$ via the kinetic Fokker-Planck equation~\eqref{eq:kinetic-FP}.
The key step is to replace the unknown score $\nabla_v\log\rho(t,\cdot)$ in the probability flow ODE by the score function estimator derived above, thereby obtaining a fully deterministic particle update. The scheme presented in this section is analogous to that in \cite{tan2026accelerated}. However, since the main focus of the present work is the derivation and analysis of the kernel formula itself, a rigorous justification of the resulting sampling algorithm is left for future work.

We recall that the kinetic Fokker-Planck equation~\eqref{eq:kinetic-FP} is the Fokker-Planck equation for the underdamped Langevin dynamics
\[
dx_t = v_t\,dt, \qquad
dv_t = -\nabla_x V(x_t)\,dt - \gamma v_t\,dt +\sqrt{2\gamma\beta^{-1}}\,dW_t.
\]
Its associated probability flow ODE, which transports the same density $\rho(t,\cdot)$ deterministically, is
\begin{equation}\label{eq:KL-ODE}
dx_t = v_t\,dt, \qquad
dv_t = -\nabla_x V(x_t)\,dt - \gamma v_t\,dt -\gamma\beta^{-1}\nabla_v\log\rho(t,x_t,v_t)\,dt.
\end{equation}
The score $\nabla_v\log\rho$ is unknown and must be approximated from the particle ensemble at each step. 
We use the mass-normalized numerical kernel approximation to estimate the implicit score through $\nabla_v\log\widehat{\mathcal K}_h\rho_k\approx\nabla_v\log\rho_{k+1}$.

Using the score estimator in \eqref{eq:score-estimator-general}, and taking the velocity component of the score function gives
\begin{align}
\nabla_v\log\widehat{\mathcal K}_h{\rho}_{k}(x,v) =&-\frac{\beta}{2\gamma}\bigl(\nabla_xV(x)+\gamma v\bigr)+\frac{3\beta}{\gamma h^2}\sum_{j=1}^N\Bigl(x-x_j-\tfrac{h}{2}(v+v_j)\Bigr)\widehat w_j(x,v)\notag \\
&-\frac{\beta}{2\gamma h}\sum_{j=1}^N(v-v_j)\,\widehat w_j(x,v).
\label{eq:kinetic-score-estimator}
\end{align}

We therefore use Algorithm~\ref{alg:kinetic_sampling_scheme} as a semi-implicit discretization of \eqref{eq:KL-ODE}. The density used in the score is advanced by one step with the mass-normalized numerical kernel approximation, while the score is evaluated at the current particle location. Thus, the method is semi-implicit at the density level but explicit in the particle location and numerical implementation.

In Algorithm~\ref{alg:kinetic_sampling_scheme}, the update is deterministic, and the diffusive effect is represented through the score function term. The sampling experiments below are intended as empirical demonstrations of the numerically approximated score estimator. Some of the displayed nonconvex targets fall outside the bounded-Hessian assumptions used in Section~\ref{sec:kernel_kinetic}. A rigorous theoretical treatment of more general targets is left for future work. 

For the experiments below, we use a modified kernel score for better numerical stability. First, define the coercive weight
\[
\phi_{\mathrm c}^{\lambda}(x,v):=\lambda V(x)+\frac{\beta}{2}\left\lvert v+\gamma^{-1}\nabla_xV(x)\right\rvert^2=\phi_{\mathrm{ki}}^\lambda(x,v)+\frac{\beta}{2\gamma^2}\lvert\nabla_xV(x)\rvert^2,
\]
where $\phi_{\mathrm{ki}}^\lambda$ is defined in \eqref{eq:phi-kin} and $\psi_{\mathrm c}(x,v):= \exp\!\left(-\frac12\phi_{\mathrm c}^{\lambda}(x,v)\right)$.  The added term is to ensure  \(\psi_{\mathrm c}\) is bounded and the corresponding Gaussian normalizers are finite.  

Second, since the original $h^3$ scaling in the Gaussian kernel \eqref{eq:G0-kin} creates numerical instability, we use
\begin{equation}\label{eq:mollified-kinetic-kernel}
G^{\mathrm{mol}}_h(x,v\mid y,w):=\frac{\left(\frac{\beta}{4\pi\gamma h}\right)^{d/2}}{(2\pi a_x)^{d/2}}\exp\!\left(-\frac{1}{2a_x}\left\|x-y-\frac{h}{2}(v+w)\right\|^2-\frac{\beta}{4\gamma h}\|v-w\|^2
\right),
\end{equation}
with \(a_x:=(\gamma h^3)/(6\beta)+\tau_x^2\), where $\tau_x$ is a small positive regularization constant. For fixed \(\tau_x>0\), the mollified kernel \(G^{\mathrm{mol}}_h\) does not have the same small-\(h\) scaling as \(G_h\). Accordingly, the mollified sampling scheme is used here as a numerical regularization and is not covered by the consistency or finite-time convergence results proved above. 
For the empirical measure \(\rho_k^N=\sum_{j=1}^N\omega_j\delta_{(x_j^k,v_j^k)}\), define the exact normalizers
\begin{equation}\label{eq:mollified-kinetic-normalizer}
Z^{\mathrm{mol}}_{h,j}:=\int_{\R^{2d}}G^{\mathrm{mol}}_h(x,v\mid x_j^k,v_j^k)\psi_{\mathrm c}(x,v)\,dx\,dv.
\end{equation}
Let \(\widehat{Z}^{\mathrm{mol}}_{h,j}>0\) denote numerical approximations obtained by applying the same Laplace procedure on \(\R^{2d}\), with the covariance and phase corresponding to \(G^{\mathrm{mol}}_h\) and \(\psi_{\mathrm c}\). The mass-normalized kernel operator \(\widehat{\mathcal K}^{\mathrm{mol}}_h\), acting on \(\rho_k^N\), is then defined as in \eqref{eq:approx-raw-operator}-\eqref{eq:approx-normalized-operator}, with \(G_h\), \(\psi\), and \(\widehat Z_h(z_j)\) replaced by \(G^{\mathrm{mol}}_h\), \(\psi_{\mathrm c}\), and \(\widehat{Z}^{\mathrm{mol}}_{h,j}\), respectively. Since the resulting mass-normalization factor is independent of \((x,v)\), it does not contribute to the logarithmic score.
 
The velocity score in the experiments is obtained by differentiating \(\log(\widehat{\mathcal K}^{\mathrm{mol}}_h\rho_k^N)\). Equivalently, in \eqref{eq:kinetic-score-estimator}, we replace \(\widehat w_j\) by the normalized weights in \eqref{eq:wj-kernel} formed using \(G^{\mathrm{mol}}_h\) and \(\widehat{Z}^{\mathrm{mol}}_{h,j}\), and replace the coefficient \(3\beta/(\gamma h^2)\) of the positional-residual term by \(h/(2a_x)\); all other terms and coefficients remain unchanged.  

\begin{algorithm}[t]
\caption{Mollified kernel-regularized kinetic sampling scheme}
\label{alg:kinetic_sampling_scheme}
\begin{algorithmic}[1]
\STATE Initialize \(\rho_0^N=N^{-1}\sum_{i=1}^N\delta_{(x_i^0,v_i^0)}\). Choose \(T,h>0\) and set \(n=\lfloor T/h\rfloor\).
\FOR{$k=0,\ldots,n-1$}
    \STATE Compute approximations \(\widehat{Z}^{\mathrm{mol}}_{h,j}\) as in \eqref{eq:laplace-approx-denominator} for \(j=1,\ldots,N\).
    \STATE For each \(i\), evaluate \(s_{i,k}:=\nabla_v\log(\widehat{\mathcal K}^{\mathrm{mol}}_h\rho_k^N)(x_i^k,v_i^k)\) using \eqref{eq:kinetic-score-estimator} with the substitutions specified above.
    \STATE Update \(v_i^{k+1}=v_i^k-h\nabla_xV(x_i^k)-h\gamma v_i^k-h\gamma\beta^{-1}s_{i,k}\) and \(x_i^{k+1}=x_i^k+h\,v_i^{k+1}\).
\ENDFOR
\end{algorithmic}
\end{algorithm}

In the following experiment, each algorithm starts from the same initial cloud,
\[
x_i^0\sim\mathcal N(0,0.45^2I_2),\qquad v_i^0\sim\mathcal N(0,0.70^2I_2).
\]
The underdamped Langevin dynamics are discretized by an Euler-Maruyama step. We pick \(\beta=\lambda=1\), \(N = 512\), \(h = 0.02\), \(T=3\), \(\gamma = 1.5\), and \(\tau_x = 0.16\).

In Fig.~\ref{fig:sampling_various}, we consider the Rosenbrock-type potential
\[
V(x)=0.9\bigl((1-x_1)^2+18(x_2-x_1^2)^2\bigr),
\]
and the six-peak target
\[
V(x)=-\log\sum_{m=1}^{6}\exp\left(-\frac{\lVert x-c_m\rVert^2}{2\sigma^2}\right),\quad c_m=R\bigl(\cos\theta_m,\sin\theta_m\bigr),\quad \theta_m=\frac{\pi(m-1)}{3},
\]
where \((R,\sigma)=(2.4,0.2)\). 
 
\begin{figure}[t]
\centering
\includegraphics[width=0.24\linewidth,trim=2.5cm 1.5cm 2.5cm 1.5cm,clip]{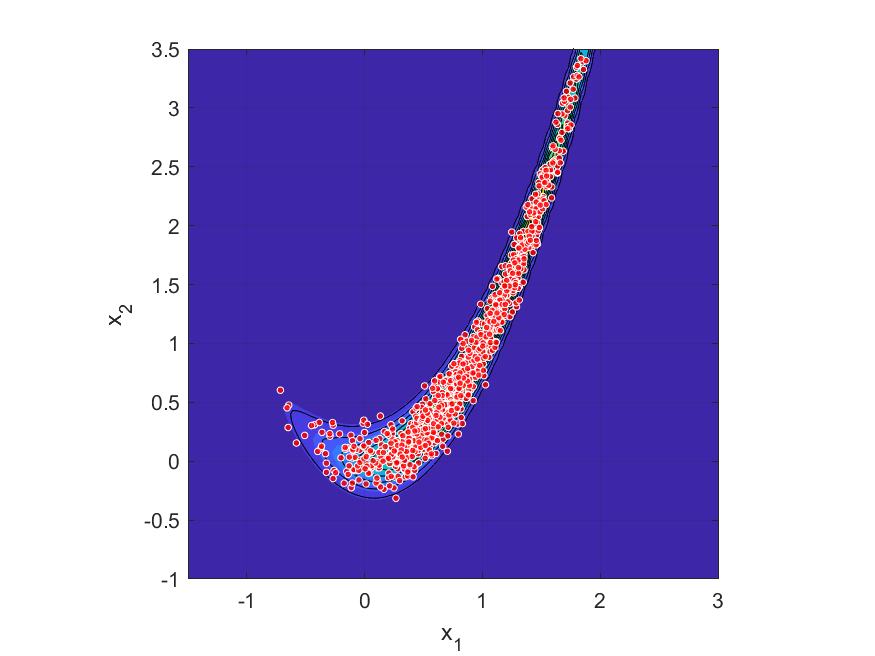}
\includegraphics[width=0.24\linewidth,trim=2.5cm 1.5cm 2.5cm 1.5cm,clip]{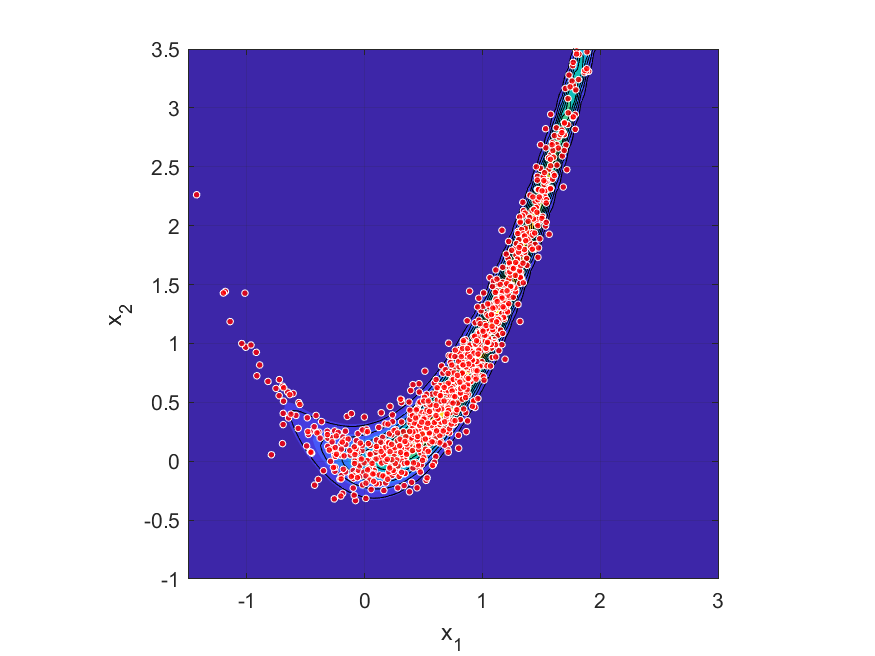}
\includegraphics[width=0.24\linewidth,trim=2.5cm 1.5cm 2.5cm 1.5cm,clip]{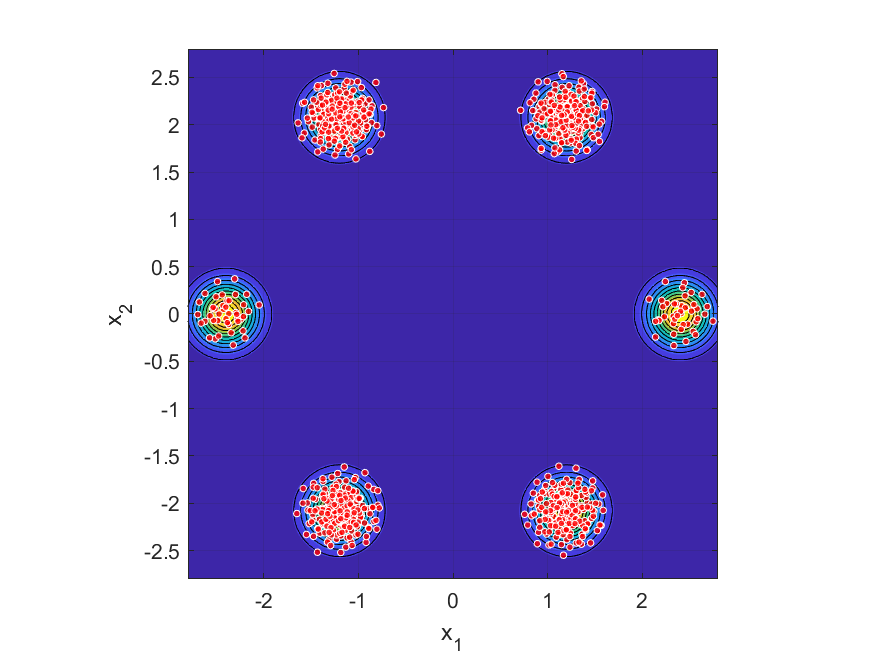}
\includegraphics[width=0.24\linewidth,trim=2.5cm 1.5cm 2.5cm 1.5cm,clip]{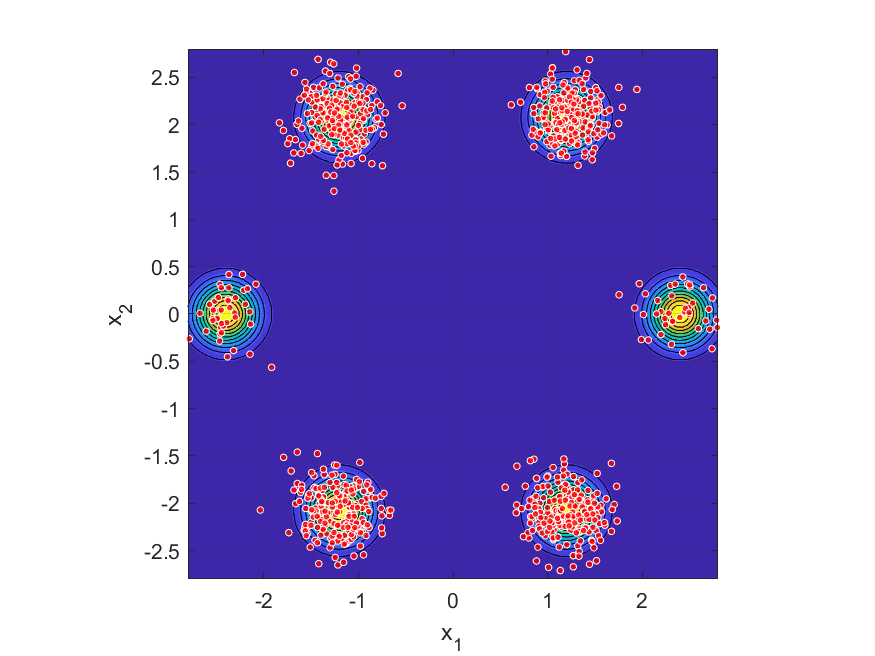}
\caption{Rosenbrock-type (left pair) and six-peak (right pair) targets.
Within each pair: particles from Algorithm~\ref{alg:kinetic_sampling_scheme} (left) and the underdamped Langevin dynamics (right).}
\label{fig:sampling_various}
\end{figure}
In Fig.~\ref{fig:sampling_various}, compared with the underdamped Langevin algorithm, the kernel scheme remains more concentrated near the Rosenbrock ridge and populates the six modes more evenly. These qualitative observations suggest that the ensemble-dependent kernel score may aid finite-time exploration, but they are not intended as a quantitative claim of superior sampling accuracy.

\section{Conclusion}
In this paper, we derived an explicit kernel formula for the kinetic Fokker-Planck equation. We established local weak consistency and, conditional on weighted backward regularity and uniform finite-time weighted-moment bounds for the exact solution and the discrete iterates, first-order finite-time weak convergence of the exact-normalization kernel scheme. We gave a formal variational characterization through a Fisher-information-regularized kinetic transport problem and recorded analogous constructions for broader classes of degenerate and nondegenerate Fokker-Planck equations. The numerical section described implementations using approximate denominators.  The numerical experiments indicate that the kernel approximation captures short-time density evolution in the tested examples and that a numerically approximated score can be used to construct deterministic particle schemes for finite-time kinetic sampling.

Beyond the present analysis, the explicit kernel formula update suggests two natural directions for future work. From the perspective of scientific computing, it may serve as a building block for efficient and structure-preserving approximation of more general parabolic and kinetic equations. From the perspective of machine learning, it may also motivate new network architectures that incorporate the interactions induced by the kernel formula and phase space evolution directly into the model design, with the longer-term goal of developing efficient kinetic generative models.

\begin{appendix}
\section{Proof of the weighted estimates in Section~\ref{sec:kernel_kinetic}}
\label{app:proof}

\begin{proof}[Proof of Lemma~\ref{lem:weighted-ratio-bound-kin}]
Let
\[
\sigma:=\sqrt{2\gamma\beta^{-1}}, \qquad \Delta_x:=hw+\sigma h^{3/2}\Bigl(\frac{1}{2\sqrt3}\xi+\frac12\eta\Bigr), \qquad \Delta_v:=\sigma h^{\frac12}\eta.
\]
The explicit Gaussian representation in \eqref{eq:G0-kin} can be rewritten as
\begin{equation}\label{eq:kinetic-gaussian-representation}
(G_h g)(y,w)=\frac{1}{(2\pi)^d}\int_{\R^d}\int_{\R^d}e^{-\frac12(|\xi|^2+|\eta|^2)}g(y+\Delta_x,w+\Delta_v)\,d\xi\,d\eta.
\end{equation}
Applying this with \(g=\psi\) gives
{\small\begin{equation}\label{eq:ratio-integral-kin}
\frac{G_h\psi(y,w)}{\psi(y,w)}=\frac{1}{(2\pi)^d}\int_{\R^d}\int_{\R^d}\exp\!\left(-\frac{|\xi|^2+|\eta|^2}{2}-\frac{\phi_{\mathrm{ki}}^\lambda(y+\Delta_x,w+\Delta_v)-\phi_{\mathrm{ki}}^\lambda(y,w)}{2}\right)\,d\xi\,d\eta.
\end{equation}}

Since \(\|D^2V\|_{L^\infty}<\infty\), the mean-value theorem yields
\[
|\nabla_xV(y)|\le C(1+|y|),
\]
and the second-order Taylor expansion implies
\[
|V(y+\zeta)-V(y)| \le C\bigl((1+|y|)|\zeta|+|\zeta|^2\bigr), \qquad |\nabla_xV(y+\zeta)-\nabla_xV(y)| \le C|\zeta|.
\]
Using the definition of \(\phi_{\mathrm{ki}}^\lambda\) in \eqref{eq:phi-kin}, we obtain
\begin{align*}
\phi_{\mathrm{ki}}^\lambda(y+\Delta_x,w+\Delta_v)-\phi_{\mathrm{ki}}^\lambda(y,w)
=&\ \lambda\bigl(V(y+\Delta_x)-V(y)\bigr)+\frac{\beta}{2}\bigl(|w+\Delta_v|^2-|w|^2\bigr) \\
&+\frac{\beta}{\gamma}\Bigl[(w+\Delta_v)\cdot\nabla_xV(y+\Delta_x)-w\cdot\nabla_xV(y)\Bigr].
\end{align*}
Since
\[
|\Delta_x|\le Ch(1+|w|)+Ch^{3/2}(|\xi|+|\eta|), \qquad |\Delta_v|\le Ch^{\frac12}|\eta|,
\]
it follows that
\begin{align}\label{eq:phi-increment-kin}
&\bigl|\phi_{\mathrm{ki}}^\lambda(y+\Delta_x,w+\Delta_v)-\phi_{\mathrm{ki}}^\lambda(y,w)\bigr|
\\ \le& Ch(1+|y|^2+|w|^2)+Ch^{\frac12}(1+|y|+|w|)|\eta|+Ch(|\xi|^2+|\eta|^2).\notag
\end{align}

Substituting \eqref{eq:phi-increment-kin} into \eqref{eq:ratio-integral-kin} gives
\begin{align*}
&\frac{G_h\psi(y,w)}{\psi(y,w)}
 \le Ce^{Ch(1+|y|^2+|w|^2)} \\
&\qquad\int_{\R^d}\int_{\R^d}\exp\!\left(-\frac{|\xi|^2+|\eta|^2}{2}+Ch^{\frac12}(1+|y|+|w|)|\eta|+Ch(|\xi|^2+|\eta|^2)\right)\,d\xi\,d\eta.
\end{align*}
For \(h\le h_0\) sufficiently small, the term \(Ch(|\xi|^2+|\eta|^2)\) is absorbed into the negative Gaussian. Completing the square in \(\eta\) then yields
\[
\frac{G_h\psi(y,w)}{\psi(y,w)}\le Ce^{Ch(1+|y|^2+|w|^2)}.
\]
Since $W(y,w)=e^{\theta(1+|y|^2+|w|^2)}$, choosing \(h_0>0\) so that \(Ch_0<\theta\) gives
\[
\sup_{0<h\le h_0}\left\|\frac{G_h\psi}{\psi}\right\|_{L_W^\infty}\le C.
\]

For the inverse ratio, fix \(\varepsilon>0\). By Young's inequality,
\[
Ch^{\frac12}(1+|y|+|w|)|\eta|\le \varepsilon(1+|y|^2+|w|^2)+C_\varepsilon h|\eta|^2.
\]
Hence \eqref{eq:phi-increment-kin} implies
\[
\phi_{\mathrm{ki}}^\lambda(y+\Delta_x,w+\Delta_v)-\phi_{\mathrm{ki}}^\lambda(y,w)\le (\varepsilon+Ch)(1+|y|^2+|w|^2)+C_\varepsilon h|\eta|^2+Ch(|\xi|^2+|\eta|^2).
\]
Restricting \eqref{eq:ratio-integral-kin} to the set \(\{|\xi|\le 1,\ |\eta|\le 1\}\), we get
\[
\phi_{\mathrm{ki}}^\lambda(y+\Delta_x,w+\Delta_v)-\phi_{\mathrm{ki}}^\lambda(y,w)\le (\varepsilon+Ch)(1+|y|^2+|w|^2)+C_\varepsilon h+Ch.
\]
Choosing \(h_0>0\) smaller if necessary so that \(Ch_0\le \varepsilon\), and enlarging \(C_\varepsilon\) if necessary, we obtain
\[
\phi_{\mathrm{ki}}^\lambda(y+\Delta_x,w+\Delta_v)-\phi_{\mathrm{ki}}^\lambda(y,w)\le 2\varepsilon(1+|y|^2+|w|^2)+C_\varepsilon h.
\]
Hence
\[
\frac{G_h\psi(y,w)}{\psi(y,w)}\ge c\,\exp\!\Bigl(-2\varepsilon(1+|y|^2+|w|^2)-C_\varepsilon h\Bigr),
\]
and therefore
\[
\frac{\psi(y,w)}{G_h\psi(y,w)}\le C\,\exp\!\Bigl(2\varepsilon(1+|y|^2+|w|^2)+C_\varepsilon h\Bigr).
\]
Choosing \(2\varepsilon<\theta\) and then \(h_0>0\) sufficiently small, we conclude that
\[
\sup_{0<h\le h_0}\left\|\frac{\psi}{G_h\psi}\right\|_{L_W^\infty}\le C.
\]
\end{proof}

\begin{proof}[Proof of Lemma~\ref{lem:weighted-Taylor-psif-kin}]
We divide the proof into two steps.

\medskip
\noindent\emph{Step 1: bound on \(D_0^2(\psi f)\).}
Since $D_0=v\cdot\nabla_x+\gamma\beta^{-1}\Delta_v$, the operator $D_0^2$ contains at most second derivatives in $x$, at most fourth derivatives in $v$, and mixed derivatives of total order at most $3$. Expanding $D_0^2(\psi f)$ by repeated use of the product rule yields
\[
\frac{D_0^2(\psi f)}{\psi}=\sum_{|\mu|\le 2,\ |\nu|\le 4,\ |\mu|+|\nu|\le 4} a_{\mu,\nu}(x,v)\,\partial_x^\mu\partial_v^\nu f(x,v),
\]
where each coefficient $a_{\mu,\nu}$ is a polynomial expression in $v$ and in derivatives of $V$ up to order $3$.
By Assumption~\ref{ass:kinetic-structure}, all such coefficients have at most polynomial growth. Since \(f\in C_W^4\), there exists an integer \(m\ge0\) such that
\begin{equation}\label{eq:D0sq-psif-bound-kin}
\left|\frac{D_0^2(\psi f)(x,v)}{\psi(x,v)}\right|\le C\,W_m(x,v)\|f\|_{C_W^4}, \qquad f\in C_W^4(\R^{2d}).
\end{equation}

\medskip
\noindent\emph{Step 2: semigroup propagation and Taylor expansion.}
Define
\[
\widehat\psi_m(x,v):=W_m(x,v)\psi(x,v).
\]
We claim that
\begin{equation}\label{eq:Gh-psim-bound-kin}
G_s\widehat\psi_m(x,v)\le C\,W_{m}^{+}(x,v)\psi(x,v), \qquad 0<s\le h_0.
\end{equation}
Using the Gaussian representation \eqref{eq:kinetic-gaussian-representation}, we have
\[
G_s\widehat\psi_m(x,v)=\frac{1}{(2\pi)^d}\int_{\R^d}\int_{\R^d} e^{-\frac12(|\xi|^2+|\eta|^2)}\widehat\psi_m(x+\Delta_x^s,v+\Delta_v^s)\,d\xi\,d\eta,
\]
where
\[
\Delta_x^s:=sv+\sigma s^{3/2}\Bigl(\frac{1}{2\sqrt3}\xi+\frac12\eta\Bigr),
\qquad
\Delta_v^s:=\sigma s^{1/2}\eta.
\]
Set \(Q(x,v):=1+|x|^2+|v|^2\) and \(r:=(|\xi|^2+|\eta|^2)^{1/2}\). The increment estimates used in the proof of Lemma~\ref{lem:weighted-ratio-bound-kin}, together with the definition of \(W_m\), give, uniformly for \(0<s\le h_0\),
\[
\frac{\widehat\psi_m(x+\Delta_x^s,v+\Delta_v^s)}{\widehat\psi_m(x,v)}\le C\bigl(1+s^{m/2}r^m\bigr)\exp\!\left(CsQ(x,v)+C\sqrt{sQ(x,v)}\,r+Csr^2\right).
\]
After decreasing \(h_0\) if necessary, Young's inequality absorbs the \(r\)-dependent exponential and polynomial factors into the Gaussian in \eqref{eq:kinetic-gaussian-representation}. Consequently,
\[
G_s\widehat\psi_m(x,v)\le C\widehat\psi_m(x,v)e^{CsQ(x,v)}.
\]
Decreasing \(h_0\) once more so that \(Ch_0<\bar\theta-\theta\), we have \(W_m e^{CsQ}\le W_m^+\) for \(0<s\le h_0\), which proves \eqref{eq:Gh-psim-bound-kin}.

Now set \(g:=\psi f\). After enlarging \(m\) if necessary, direct differentiation gives
\[
|g|+|D_0g|+|D_0^2g|\le C\widehat\psi_m\|f\|_{C_W^4}.
\]
We justify the semigroup Taylor identity by localization. Let \(\chi\in C_c^\infty(\R^{2d})\) satisfy \(0\le\chi\le1\) and \(\chi=1\) on the unit ball, and set \(\chi_R(z):=\chi(z/R)\) and \(g_R:=\chi_Rg\). Since \(g_R\) is compactly supported and has the derivatives required by \(D_0^2\), Dynkin's formula applied twice gives
\[
G_hg_R-g_R-hD_0g_R=\int_0^h(h-s)G_sD_0^2g_R\,ds.
\]
The product rule, the derivative bounds used in Step~1, and a further enlargement of \(m\) give, uniformly for \(R\ge1\),
\[
|g_R|+|D_0g_R|+|D_0^2g_R|\le C\widehat\psi_m\|f\|_{C_W^4}.
\]
Therefore \eqref{eq:Gh-psim-bound-kin} and dominated convergence allow \(R\to\infty\), yielding
\[
G_hg-g-hD_0g=\int_0^h(h-s)G_sD_0^2g\,ds.
\]
Finally, positivity of \(G_s\), \eqref{eq:D0sq-psif-bound-kin}, and \eqref{eq:Gh-psim-bound-kin} imply
\[
\left|\frac{G_sD_0^2(\psi f)(x,v)}{\psi(x,v)}\right|\le\frac{G_s|D_0^2(\psi f)|(x,v)}{\psi(x,v)}\le C\|f\|_{C_W^4}\frac{G_s\widehat\psi_m(x,v)}{\psi(x,v)}\le C W_m^+(x,v)\|f\|_{C_W^4}.
\]
Dividing by \(W_{m}^{+}(x,v)\), taking the supremum, and integrating in \(s\) yields \eqref{eq:weighted-Taylor-psif-kin}. The particular case \eqref{eq:weighted-Taylor-psi-kin} follows by taking \(f\equiv 1\).
\end{proof}
\section{Numerical details.} Details for \(\rho_0\) in (61).
The mixture weights are
\[
(p_1,\ldots,p_9)=(0.10,0.09,0.11,0.12,0.16,0.10,0.11,0.09,0.12).
\]
The component means, ordered as \(m_\ell=(x_\ell,v_\ell)^\top\), are
\[
\begin{aligned}
&m_1=(-2.4,-2.1)^\top, \qquad
m_2=(-2.5,0)^\top, \qquad
m_3=(-2.4,2.1)^\top,\\
&m_4=(0,-2.0)^\top, \qquad
m_5=(0,0)^\top, \qquad
m_6=(0,2.0)^\top,\\
&m_7=(2.4,-2.1)^\top, \qquad
m_8=(2.5,0)^\top, \qquad
m_9=(2.4,2.1)^\top.
\end{aligned}
\]
The component covariance matrices are
\[
\begin{aligned}
\Sigma_1&=
\begin{pmatrix}
0.080 & 0.010\\
0.010 & 0.070
\end{pmatrix},
&
\Sigma_2&=
\begin{pmatrix}
0.090 & -0.015\\
-0.015 & 0.075
\end{pmatrix},
&
\Sigma_3&=
\begin{pmatrix}
0.075 & 0.012\\
0.012 & 0.085
\end{pmatrix},
\\[6pt]
\Sigma_4&=
\begin{pmatrix}
0.085 & -0.010\\
-0.010 & 0.070
\end{pmatrix},
&
\Sigma_5&=
\begin{pmatrix}
0.095 & 0\\
0 & 0.090
\end{pmatrix},
&
\Sigma_6&=
\begin{pmatrix}
0.080 & 0.018\\
0.018 & 0.080
\end{pmatrix},
\\[6pt]
\Sigma_7&=
\begin{pmatrix}
0.070 & -0.012\\
-0.012 & 0.080
\end{pmatrix},
&
\Sigma_8&=
\begin{pmatrix}
0.085 & 0.014\\
0.014 & 0.075
\end{pmatrix},
&
\Sigma_9&=
\begin{pmatrix}
0.078 & -0.008\\
-0.008 & 0.082
\end{pmatrix}.
\end{aligned}
\]
\end{appendix}

\bibliographystyle{plain}  

\bibliography{ref.bib}
 \end{document}